\documentclass[a4paper,10pt]{article}
\usepackage{fullpage}
\usepackage{amsmath,amssymb,latexsym,amsthm}
\usepackage[pdftex]{graphicx}
\usepackage{graphics}
\usepackage{psfrag}
\usepackage{color, epsfig}
\usepackage{mathrsfs}
\usepackage[frenchb,english]{babel}
\usepackage[utf8]{inputenc}

\def\R{{\mathbb R}}
\def\N{{\mathbb N}}

\def\P{{\mathbb P}}

\def\G{{\mathscr{G}}}
\def\W{{\mathscr{W}}}
\def\ds{\displaystyle}

\newcommand{\norm}[1]{\left\Vert#1\right\Vert}

\renewcommand{\leq}{\leqslant}
\renewcommand{\geq}{\geqslant}

\newtheorem{theorem}{Theorem}[section]

\newtheorem{corollary}[theorem]{Corollary}
\newtheorem{proposition}[theorem]{Proposition}
\newtheorem{lemma}[theorem]{Lemma}
\newtheorem{remark}[theorem]{Remark}
\numberwithin{equation}{section}

\title{Control issues and linear projection constraints on the control and on the controlled trajectory\footnote{This work has been supported by the Agence Nationale de la Recherche, Project IFSMACS, grant ANR-15-CE40-0010, and by the CIMI Labex, Toulouse, France, under grant ANR-11-LABX-0040-CIMI.}}
\author{
Sylvain Ervedoza\footnote{Institut de Math\'ematiques de Toulouse ; UMR5219; Universit\'e de Toulouse ; CNRS ; UPS IMT, F-31062 Toulouse Cedex 9, France, 
{\tt sylvain.ervedoza@math.univ-toulouse.fr}}
}

\date\today
\begin{document}
\maketitle
\begin{abstract}
	The goal of this article is to discuss controllability properties for an abstract linear system of the form $y' = A y + B u$ under some additional linear projection constraints on the control $u$ or / and on the controlled trajectory $y$. In particular, we discuss the possibility of imposing the linear projections of the controlled trajectory and of the control, in the context of approximate controllability, exact controllability and null-controllability. As it turns out, in all these settings, for being able to impose linear projection constraints on the control or/and the controlled trajectory, we will strongly rely on a unique continuation property for the adjoint system which, to our knowledge, has not been identified so far, and which does not seem classical. We shall therefore provide several instances in which this unique continuation property can be checked.  
\end{abstract}
%
%
%
%
%%%%%%%%%%%%%%%%%%%%%%%%%%%%%%%%%%%%%%%%%%%%%%%%%%%%%%%%%%%%%%%%%%%%%%%%%%%%%%%%%%%%%%%%%%%%%%%%%%%%%%%%%%%%%%%%%%%%%%%%%%%%%%%%%%%%%%%
%
%
%%%%%%%%%%%%%%%%%%%%%%%%%%%%%%%%%%%%%%%%%%%%%%%%%%%%%%%%%
%
%
%
%
\section{Introduction}
%
%
%%%%%%%%%%%%%%%%%%%%%%%%%%%%%%%%%%%%%%%%%%%%%%%%%%%%%%%%%
%
The goal of this article is to study controllability issues for an abstract system of the form 
\begin{equation}
	\label{ControlledEq}
		y'  = A y + B u, \quad \hbox{ for } t \in (0,T), \qquad y(0) = y_0.
\end{equation}
Let us make precise the functional setting we shall consider in the following:
\begin{itemize}
	\item[(H1)] $A$ is assumed to generate a $C_0$ semigroup on a Hilbert space $H$,
	\item[(H2)] $B$ is the control operator, assumed to belong to $\mathscr{L}(U; H)$, where $U$ is a Hilbert space.
\end{itemize}
The function $y = y(t)$ is then the state function, $y_0$ is the initial datum, and $u$ is the control function, assumed to belong to $L^2(0,T; U)$.

Note that, within these assumptions, if $y_0 \in H$ and $u \in L^2(0,T; U)$, the solution $y$ of \eqref{ControlledEq} belongs to $\mathscr{C}^0([0,T]; H)$.

In this article, we wish to understand the requirements needed to be able to control the state $y$ solving \eqref{ControlledEq} and to impose the linear projections on $y$ and / or on $u$ and / or $y(T)$ in some vector spaces. In particular, we shall consider the following setting:
\begin{itemize}
	\item[(H3)] $\G$ is a closed vector space of $L^2(0,T; U)$, and $\P_\G$ is the orthogonal projection on $\G$ in $L^2(0,T; U)$.
	\item[(H4)] $\W$ is a closed vector space of $L^2(0,T; H)$, and $\P_\W$ is the orthogonal projection on $\W$ in $L^2(0,T; H)$.
	\item[(H5)] $E$ is a finite dimensional space of $H$, and $\P_E$ is the orthogonal projection on $E$ in $H$.
\end{itemize}

We shall then discuss the following properties.
\\
{\bf Approximate controllability and linear projection constraints:} For $y_0 \in H$ and $y_1 \in H$, $\varepsilon >0$, $g_* \in \G$, $w_* \in \W$, can we find control functions $u \in L^2(0,T; U)$ such that 
\begin{equation}
	\label{Moments-Of-U-in-G}
		\P_\G u = g_*
\end{equation}
and the solution $y$ of \eqref{ControlledEq} satisfies
\begin{equation}
	\label{ApproxCont-Req}
		\norm{y(T) - y_1}_H \leq \varepsilon, 
\end{equation}
and 
\begin{equation}
	\label{Moments-Of-Y-in-W}
		\P_\W y = w_*, 
\end{equation}
and
\begin{equation}
	\label{Proj-Of-Y-T-in-E}
		\P_E y(T) = \P_E y_1 \ ?
\end{equation}
{\bf Exact controllability and linear projection constraints:} For $y_0 \in H$ and $y_1 \in H$, $g_* \in \G$ and $w_* \in \W$, can we find control functions $u \in L^2(0,T; U)$ such that \eqref{Moments-Of-U-in-G} holds, and the solution $y$ of \eqref{ControlledEq} satisfies
\begin{equation}
	\label{ExactCont-Req}
		y(T)= y_1, 
\end{equation}
and \eqref{Moments-Of-Y-in-W}?
\\
{\bf Null controllability and linear projection constraints:} For $y_0 \in H$, $g_* \in \G$ and $w_* \in \W$, can we find control functions $u \in L^2(0,T; U)$ such that \eqref{Moments-Of-U-in-G} holds, and the solution $y$ of \eqref{ControlledEq} satisfies
\begin{equation}
	\label{NullCont-Req}
		y(T)= 0, 
\end{equation}
and \eqref{Moments-Of-Y-in-W}?
Of course, the above problems correspond to reinforcements of the classical notions of approximate controllability, exact controllability and null controllability, for which we refer to the textbook \cite{TWBook}. To be more precise, the only originality in the above notions lies in the conditions \eqref{Moments-Of-U-in-G} and \eqref{Moments-Of-Y-in-W} on the respective projections of $u$ on $\G$ and $y$ on $\W$. This question appeared to be of interest in some control problems, for instance in order to ensure insensibility with respect to some parameters, see e.g. the book \cite{Lions-Sentinelles-1992}. We will come back to these questions later on an example inspired by previous works \cite{Nakoulima-2004,Mophou-Nakoulima-2008,Mophou-Nakoulima-2009,Peng-2015}. In fact, our interest in this question was triggered by our work \cite{Chowdhury-Erv-2019}, in which at some part of the proof, we needed to derive controls satisfying appropriate projection constraints (there, we managed to build such controls by using some null-controllability results and the structure of the constraints we wanted to impose), and by the recent work \cite{Lissy-Privat-Simpore} on insensibility with respect to variations of the domain. 
%

%%%%%%%%%%%%%%%%%%%%%%%%%%%%%%%%%%%%%%%%%%%%%%%%%%%%%%%%%
%
%
%
%
%
%
%%%%%%%%%%%%%%%%%%%%%%%%%%%%%%%%%%%%%%%%%%%%%%%%%%%%%%%%%
%
%
Let us start with the problem of approximate controllability and linear projection constraints. 
\begin{theorem}[Approximate controllability with linear projection constraint]
	\label{Thm-ApproxControl}
	Let the hypotheses (H1)--(H5) be satisfied, and let $T>0$.
	\\
	Assume the following unique continuation property: If, for some $z_T \in H$, for some $g \in \G$ and $w \in \W$, the solution $z$ of 
	\begin{equation}
		\label{AdjointEq}
		z' +  A^* z =  w , \quad \hbox{ for } t \in (0,T), \quad \text{ with } z(T) = z_T, 
	\end{equation}
	satisfies
	\begin{equation}
		\label{LinearRelation}
		B^* z =  g  \quad \hbox{ in } (0,T), 
	\end{equation}
	then 
	\begin{equation}
		\label{Unique-Continuation}
		z_T = 0, 
		\hphantom{ and } \quad g = 0,
		\quad \hbox{ and } \quad w = 0.
	\end{equation}
	Assume moreover that the vector space $\W$ is of finite dimension.
	\\
	Then for any $y_0$ and $y_1$ in $H$, $\varepsilon >0$, $g_* \in \G$, and $w_* \in \W$, there exists a control function $u \in L^2(0,T; U)$ such that \eqref{Moments-Of-U-in-G} holds, the solution $y$ of \eqref{ControlledEq} satisfies \eqref{ApproxCont-Req}, and the conditions \eqref{Moments-Of-Y-in-W} and \eqref{Proj-Of-Y-T-in-E}.
	\\
	In other words, one can solve the approximate controllability problem and exactly satisfy the linear projection constraints \eqref{Moments-Of-U-in-G} on $u$, \eqref{Moments-Of-Y-in-W} on $y$, and the constraint \eqref{Proj-Of-Y-T-in-E} on $y(T)$.
\end{theorem}

%
%\begin{remark}
%	\label{Rem-W-Inf-Dim-Approx}
%	%
%	In fact, the condition that $\W$ is of finite dimension could be replaced by the following one: If $w_{n}$ denotes a sequence of elements of $\W$ which is weakly convergent to $0$ in $L^2(0,T; H)$, then the trajectories $z_n$ defined by 
%	%
%	$$
%		z_n' +A^* z_n = w_n, \quad t \in (0,T), \qquad \hbox{ with } z_n(T) = 0
%	$$
%	%
%	strongly converge to $0$ in $L^2(0,T; H)$. 
%	%
%	\\
%	%
%	We refer the interested reader to Remark \ref{Rem-W-Inf-Dim-Approx-Proof} for the explanation of the fact that this condition is sufficient.
%	%
%\end{remark}

Theorem \ref{Thm-ApproxControl} is proved in Section \ref{Sec-Proof-Approx}. 

Before going further, several remarks are in order. 

First, let us recall that it is well-known (see e.g. \cite[Theorem 11.2.1]{TWBook}) that approximate controllability (that is, the problem of, for any $\varepsilon >0$, $y_0$ and $y_1$ in $H$, finding a control function $u \in L^2(0,T; U)$ such that the solution $y$ of \eqref{ControlledEq} satisfies \eqref{ApproxCont-Req}) is equivalent to the following unique continuation property: 
\begin{equation}
	\label{UC-Classical}
	\hbox{If } z \in \mathscr{C}^0([0,T]; H) \hbox{ satisfies } 
	\left\{ 
		\begin{array}{ll}
		 z' + A^* z = 0, \quad & \hbox{ for }t \in (0,T),
		\\
		z(T) = z_T, 
		\\
		B^* z = 0 ,  \quad & \hbox{ for } t \in (0,T),
		\\
		\hbox{ with } z_T  \in H
		\end{array} 
	\right.	
	\qquad \hbox{ then } \qquad
		z_T = 0.
\end{equation}
In this sense, the unique continuation property assumed in Theorem \ref{Thm-ApproxControl}, namely
\begin{equation}
	\label{UC}
	\tag{UC}
	\hbox{If } z \hbox{ satisfies }
	\left\{ 
		\begin{array}{ll}
		 z' + A^* z = w, \quad & \hbox{ for } t \in (0,T),
		\\
		z(T) = z_T, 
		\\
		B^* z = g ,  \quad &  \hbox{ for } t \in (0,T),
		\\
		\lefteqn{\hbox{ with } (z_T, g, w) \in H \times \G \times \W}
		\end{array} 
	\right.	
	\qquad \hbox{ then } \qquad
	\left\{ 
		\begin{array}{l}
		z_T = 0,
		\\
		g = 0, 
		\\
		w = 0,
		\end{array} 
	\right.	
\end{equation}
 is a stronger version of the  standard unique continuation property \eqref{UC-Classical} for the adjoint equation \eqref{AdjointEq}. 

It is therefore quite natural to ask if the unique continuation property \eqref{UC} assumed in Theorem~\ref{Thm-ApproxControl} is sharp or not. We claim that this is indeed the sharp condition. Indeed, if there exists some non-zero $(z_T, g, w) \in H \times \G \times \W$ such that \eqref{AdjointEq} and \eqref{LinearRelation} holds, then one easily checks that, for any $u  \in L^2(0,T; U)$, the solution $y$ of \eqref{ControlledEq} necessarily satisfies:
$$
	0 = \langle y(T), z_T \rangle_H - \langle y_0, z(0) \rangle_H
	- 
	\int_0^T \langle y(t),  w(t) \rangle_H \, dt
	- 
	\int_0^T \langle u(t),  g(t) \rangle_U \, dt. 
$$
In particular, if one wishes to impose $\P_\W y = w$, $\P_\G u = g$ for a solution $y$ of \eqref{ControlledEq} starting from $y_0 =  0$, we deduce that necessarily, 
$$
	\norm{y(T)+z_T}_H \norm{z_T}_H \geq \norm{z_T}_H^2  +  \| w \|_{L^2(0,T; H)}^2 + \| g \|_{L^2(0,T; U)}^2.
$$
Since the above right hand-side is strictly positive by assumption, this implies that there exists a neighborhood of $-z_T$ such that the trajectories $y$ of \eqref{ControlledEq} starting from $y_0 = 0$ cannot reach this neighborhood and satisfy the constraints $\P_\W y = w$ and $\P_\G u = g$.
It might be surprising at first that the unique continuation property \eqref{UC} in Theorem \ref{Thm-ApproxControl} does not depend on the vector space $E$ appearing in condition \eqref{Proj-Of-Y-T-in-E}. In fact, it was already noticed in \cite{Zuazua-97,FernandezCaraZuazua1} in the case of $\G = \{0\}$ and $\W = \{0\}$, that the usual unique continuation property is sufficient to solve the approximate controllability problem with the constraint \eqref{Proj-Of-Y-T-in-E}. This property strongly relies on the fact that $E$ is finite dimensional (recall that it is part of assumption (H5)).
The unique continuation property \eqref{UC} may not seem easy to check in practice. We will however give several examples on which this can be checked out, one which is in fact the one in \cite{Nakoulima-2004,Mophou-Nakoulima-2008,Mophou-Nakoulima-2009}, and another one which is inspired by the one in \cite{Chowdhury-Erv-2019}. The interested reader can go directly to Section \ref{Sec-Examples}. 
An interesting point is that our approach can in fact be developed as well for the other notions of controllability stated in the introduction, namely the exact controllability problem with linear projection constraints and the null-controllability problem with linear projection constraints. 

\begin{theorem}[Exact controllability with linear projection constraints]
	\label{Thm-ExactCont}
	Let the hypotheses (H1)--(H4) be satisfied, and let $T>0$, and assume the unique continuation property \eqref{UC}.
	\\
	We further assume the following observability inequality: there exists a constant $C>0$ such that for all $z_T \in H$, the solution $z$ of 
	\begin{equation}
		\label{Adjoint-Hom}
		z' + A^* z = 0, \quad \hbox{ for } t \in (0,T), \qquad z(T) = z_T, 
	\end{equation}
	satisfies
	\begin{equation}
		\label{Obs-z-T}
		\norm{z_T}_H \leq C \norm{ B^* z}_{L^2(0,T; U)}. 
	\end{equation}
	Assume moreover that the vector spaces $\G$ and $\W$ are of finite dimension.
	\\
	Then for any $y_0$ and $y_1$ in $H$, $g_* \in \G$, and $w_* \in \W$, there exists a control function $u \in L^2(0,T; U)$ such that \eqref{Moments-Of-U-in-G} holds, the solution $y$ of \eqref{ControlledEq} satisfies \eqref{ExactCont-Req} and the condition \eqref{Moments-Of-Y-in-W}.
	\\
	In other words, one can solve the exact controllability problem and exactly satisfy the constraints \eqref{Moments-Of-U-in-G} on $u$ and \eqref{Moments-Of-Y-in-W} on $y$.
\end{theorem}
The proof of Theorem \ref{Thm-ExactCont} is given in Section \ref{Sec-ExactCont}. 
Note that Theorem \ref{Thm-ExactCont} requires not only the unique continuation property \eqref{UC}, but also the observability property \eqref{Obs-z-T} for solutions of \eqref{Adjoint-Hom}. This is expected, as the usual exact controllability property (that is, the problem of, for any $y_0$ and $y_1$ in $H$, finding a control function $u \in L^2(0,T; U)$ such that the solution $y$ of \eqref{ControlledEq} satisfies \eqref{ExactCont-Req}) is equivalent to the observability property \eqref{Obs-z-T}. Here, since we would like to further impose some projections of $u$ and $y$, we should further assume the unique continuation property  \eqref{UC}, similarly as in Theorem \ref{Thm-ApproxControl}. In fact, the proof of Theorem \ref{Thm-ExactCont} given in Section \ref{Sec-ExactCont} mainly boils down to the proof of the following observability inequality (see Lemma \ref{Lem-Obs-z-T-General} and its proof in Section \ref{Subsec-Lem-Obs-z-T}): there exists $C>0$ such that for all $(z_T, g, w, f) \in H \times \G \times \W \times L^2(0,T; H)$, the solution $z$ of 
\begin{equation}
	\label{Adjoint-NonHom}
	z' + A^* z = f, \quad \hbox{ for } t \in (0,T), \quad \text{ with } \quad z(T) = z_T
\end{equation}
satisfies
\begin{equation}
	\label{Obs-z-T-General}
	\| ( z_T, g, w, f)\|_{H \times \G \times \W \times L^2(0,T; H)}
	\leq 
	C \left( \norm{ B^* z + g}_{L^2(0,T; U)} + \norm{ f + w}_{L^2(0,T; H)} \right).
\end{equation}
This is actually at this step that we strongly use the fact that the vector spaces $\G$ and $\W$ are of finite dimension, which allows to deduce the observability inequality \eqref{Obs-z-T-General} for solutions of \eqref{Adjoint-NonHom} from a compactness argument based on the unique continuation property \eqref{UC} and the observability inequality \eqref{Obs-z-T} for solutions of \eqref{Adjoint-Hom}. In fact, if we consider vector spaces $\G$ and $\W$ of possibly infinite dimension, our proof of Theorem \ref{Thm-ExactCont} yields the following result, whose detailed proof is left to the reader as it is a verbatim copy of Section \ref{Subsec-Strategy-ExactCont}:
\begin{corollary}
	\label{Cor-ExactCont-Infinite}
	Let the hypotheses (H1)--(H4) be satisfied, and let $T>0$, and assume the observability inequality \eqref{Obs-z-T-General} for solutions of \eqref{Adjoint-NonHom}. 
	\\
	Then, for any $y_0$ and $y_1$ in $H$, $g_* \in \G$, and $w_* \in \W$, there exists a control function $u \in L^2(0,T; U)$ such that \eqref{Moments-Of-U-in-G} holds, the solution $y$ of \eqref{ControlledEq} satisfies \eqref{ExactCont-Req} and the condition \eqref{Moments-Of-Y-in-W}.
\end{corollary}

Similarly, when considering null-controllability with linear projection constraints, one should rely on some kind of observability properties for solutions of \eqref{Adjoint-NonHom}:

\begin{theorem}[Null controllability with linear projection constraints]
	\label{Thm-NullCont}
	Let the hypotheses (H1)--(H4) be satisfied, and let $T>0$.
	\\
	We further assume the following observability inequality: there exists a constant $C>0$ such that for all $(z_T, g, w, f) \in H \times \G \times \W \times L^2(0,T; H)$, the solution $z$ of  \eqref{Adjoint-NonHom} satisfies 
	\begin{equation}
		\label{Obs-z-0-General}
		\| ( z(0), g, w, f)\|_{H \times \G \times \W \times L^2(0,T; H)}
		\leq 
		C \left( \norm{ B^* z + g}_{L^2(0,T; U)} + \norm{ f + w}_{L^2(0,T; H)} \right).
	\end{equation}
	Then, for any $y_0$ in $H$, $g_* \in \G$, and $w_* \in \W$, there exists a control function $u \in L^2(0,T; U)$ such that \eqref{Moments-Of-U-in-G} holds, the solution $y$ of \eqref{ControlledEq} satisfies \eqref{NullCont-Req} and the condition \eqref{Moments-Of-Y-in-W}.
	\\
	In other words, one can solve the null controllability problem \eqref{NullCont-Req} and exactly satisfy the constraints \eqref{Moments-Of-U-in-G} on $u$ and \eqref{Moments-Of-Y-in-W} on $y$.
\end{theorem}
The proof of Theorem \ref{Thm-NullCont} is given in Section \ref{Subsec-Proof-NullCont}, and is quite similar to the one of Corollary \ref{Cor-ExactCont-Infinite}. 

Let us point out that Theorem \ref{Thm-NullCont} relies on the observability property \eqref{Obs-z-0-General} for solutions of \eqref{Adjoint-NonHom}, which is the counterpart of the observability property \eqref{Obs-z-T-General} for Theorem \ref{Thm-ExactCont}. Still, as in the case of exact controllability, one could ask if the observability inequality \eqref{Obs-z-0-General} for solutions of \eqref{Adjoint-NonHom} could be derived from the observability inequality which is equivalent to null-controllability, namely the following one: there exists a constant $C>0$ such that for all solutions $z$ of \eqref{Adjoint-Hom} with $z_T \in H$, 
\begin{equation}
	\label{Obs-z-0-Usual}
	\norm{z(0)}_H \leq C \norm{ B^* z}_{L^2(0,T; U)}.
\end{equation}
It is not clear whether or not the observability inequality \eqref{Obs-z-0-General} for solutions of \eqref{Adjoint-NonHom} can be derived from the observability inequality \eqref{Obs-z-0-Usual} for solutions of \eqref{Adjoint-Hom} and the unique continuation property \eqref{UC}. In fact, using a compactness argument, we only managed to obtain the following result, proved in Section \ref{Subsec-Proof-Prop}:
\begin{proposition}
	\label{Prop-From-Usual-Obs-z0-to-Obs-z0-Gal}
	Let the hypotheses (H1)--(H4) be satisfied, and let $T>0$, and assume that the vector spaces $\G$ and $\W$ are of finite dimension.
	
	We further assume that there exists $\tilde T \in (0,T]$ such that
	\begin{equation}
	\label{UC-Tilde}
	\hbox{if } z \hbox{ satisfies }
	\left\{ 
		\begin{array}{ll}
		 z' + A^* z = w, \quad & \hspace{-2cm}\hbox{ for } t \in (0,\tilde T),
		\\
		z(\tilde T) = z_{\tilde T}, 
		\\
		B^* z = g , \quad &\hspace{-2cm}\hbox{ for } t \in (0,\tilde T),
		\\
		\hbox{ with } (z_{\tilde T}, g, w) \in H \times \G \times \W,
		\end{array} 
	\right.	
	\quad \hbox{ then } \quad 
	\left\{ 
		\begin{array}{l}
		z_{\tilde T} = 0,
		\\
		g = 0, 
		\\
		w = 0,
		\end{array} 
	\right.	
	\end{equation}
	and such that there exists a constant $C$ such that any solution $z$ of \eqref{Adjoint-Hom} with $z_T \in H$ satisfies
	\begin{equation}
		\label{Obs-z-tilde-T}
		\norm{z(\tilde T)}_H \leq C \norm{B^*z}_{L^2(0,T; U)}.
	\end{equation}
	Then the observability inequality \eqref{Obs-z-0-General} holds for all solutions $z$ of \eqref{Adjoint-NonHom} with $(z_T, g, w, f) \in H \times \G \times \W \times L^2(0,T; H)$.
\end{proposition}
\begin{remark}
	We emphasize that elements of $\G$ and $\W$ are defined on the whole time interval $(0,T)$, so that the conclusion of \eqref{UC-Tilde} has to be understood as $g = 0$ and $w = 0$ in the whole time interval $(0,T)$. Therefore, for condition \eqref{UC-Tilde} to be satisfied, the sets $\G$ and $\W$ should satisfy the following conditions: 
	\begin{equation*}
		\forall g \in \G, \, g_{|(0,\tilde T)} = 0 \ \Rightarrow \ g = 0, 
		\quad \hbox{ and } \quad 
		\forall w \in \W, \, w_{|(0,\tilde T)} = 0 \ \Rightarrow \ w = 0.
	\end{equation*}
\end{remark}
Roughly speaking, Proposition \ref{Prop-From-Usual-Obs-z0-to-Obs-z0-Gal} reduces the proof of the observability inequality \eqref{Obs-z-0-General} to the existence of an intermediate time $\tilde T$ such that the unique continuation property \eqref{UC-Tilde} holds and the observability inequality \eqref{Obs-z-tilde-T} holds for solutions of \eqref{Adjoint-Hom}. Note that the unique continuation property \eqref{UC-Tilde} is slightly stronger than \eqref{UC} since the time $\tilde T$ is smaller than $T$. Similarly, the observability inequality \eqref{Obs-z-tilde-T} is slightly stronger than \eqref{Obs-z-0-Usual} since $\tilde T >0$. Also note that, if the assumptions of Proposition \ref{Prop-From-Usual-Obs-z0-to-Obs-z0-Gal} holds for $\tilde T = T$, we are in fact in the setting of Theorem \ref{Thm-ExactCont}, so that one can solve the exact controllability problem with linear projection constraints, and therefore the null-controllability problem as well. 

From the above discussions, it is clear that what plays a key role in our analysis is the unique continuation property \eqref{UC}. We shall thus provide some examples in which it can be checked, see Section \ref{Sec-Examples}.

Let us finally mention that, in the cases $\G= \{0\}$ and $\W = \{0\}$, the assumption \eqref{UC} in Theorem \ref{Thm-ApproxControl} is a necessary and sufficient condition for the approximate controllability of \eqref{ControlledEq}; similarly, the observability condition \eqref{Obs-z-T} for solutions of \eqref{Adjoint-Hom} in Theorem \ref{Thm-ExactCont} is a necessary and sufficient condition for the exact controllability of \eqref{ControlledEq}, and the observability condition \eqref{Obs-z-0-General} for solutions of \eqref{Adjoint-NonHom} in Theorem \ref{Thm-NullCont} also is a necessary and sufficient condition for the null controllability of \eqref{ControlledEq}. We refer, for instance, to the textbook \cite[Theorem 11.2.1]{TWBook} for the proof of these results.

\medskip

{\bf Outline.} Section \ref{Sec-Proof-Approx} analyzing the approximate controllability problem \eqref{Moments-Of-U-in-G}--\eqref{ApproxCont-Req}--\eqref{Moments-Of-Y-in-W}--\eqref{Proj-Of-Y-T-in-E} provides the proof of Theorem \ref{Thm-ApproxControl}. Our result on exact controllability, namely Theorem \ref{Thm-ExactCont}, is proven in Section \ref{Sec-ExactCont}. Theorem \ref{Thm-NullCont} and Proposition \ref{Prop-From-Usual-Obs-z0-to-Obs-z0-Gal} discussing the null-controllability problem \eqref{Moments-Of-U-in-G}, \eqref{Moments-Of-Y-in-W} and \eqref{NullCont-Req} are then proved in Section \ref{Sec-NullCont}. In Section \ref{Sec-Examples}, we provide several PDE examples in which the crucial unique continuation property \eqref{UC} can be checked. Finally, we give some further comments and open problems in Section \ref{Sec-Further}.

\medskip

{\bf Acknowledgments. } This work benefited from various discussions with colleagues. The author is particularly indebted to Jean-Pierre Raymond for having pointed out the works \cite{Nakoulima-2004,Mophou-Nakoulima-2008,Mophou-Nakoulima-2009}, to J\'er\'emi Dard\'e, Pierre Lissy and Yannick Privat for their strong encouragements and comments, and the warmly atmosphere of the Pau Toulouse workshop in shape optimization, in which a first version of this work was presented. 

%
%%%%%%%%%%%%%%%%%%%%
%
\section{Approximate controllability: Proof of Theorem \ref{Thm-ApproxControl}}
\label{Sec-Proof-Approx}

\subsection{Main steps of the proof of Theorem \ref{Thm-ApproxControl}} 

We assume (H1)--(H5), we take $T>0$ and assume the unique continuation property \eqref{UC}. We then set $(y_0, y_1) \in H^2$, $\varepsilon >0$, $g_* \in \G$ and $w_* \in \W$. 

The proof of Theorem \ref{Thm-ApproxControl} relies on the introduction of the functional 
\begin{multline}
	\label{Functional-Approx-Cont}
		J_{\text{ap}}(z_T, g,w , f )
		= 
		\frac{1}{2} \int_0^T \norm{ B^* z(t) + g(t) }_{U}^2 \, dt 
		+ 
		\frac{1}{2} \int_0^T \norm{ f(t) + w(t)}_H^2 \, dt
		+
		\langle y_0, z(0) \rangle_H - \langle y_1, z_T \rangle_H
		\\
		+ 
		\int_0^T \langle B^* z(t), g_*(t) \rangle_U \, dt
		+ 
		 \int_0^T \langle f(t), w_*(t) \rangle_H \, dt
		+ 
		\varepsilon \norm{ (I - \mathbb{P}_E) z_T}_H, 
\end{multline}
defined for 
$$
	(z_T, g,w , f ) \in H \times \G \times \W \times L^2(0,T; H),
$$
where $z$ denotes the solution of \eqref{Adjoint-NonHom}.

Namely, we shall distinguish two main steps. The first step consists in showing that $J_{\text{ap}}$ is coercive, and the second one in proving that the minimizer provides a solution to the control problem \eqref{Moments-Of-U-in-G}, \eqref{ApproxCont-Req}, \eqref{Moments-Of-Y-in-W} and \eqref{Proj-Of-Y-T-in-E}. The corresponding statements are given by the following lemmas, whose respective proofs are given in the section afterwards.
\begin{lemma}
	\label{Lem-Coercivity-of-J}
	The functional $J_{\text{ap}}$ is continuous, strictly convex and coercive in $ H \times \G \times \W \times L^2(0,T; H)$.
\end{lemma}
Of course, based on Lemma \ref{Lem-Coercivity-of-J}, the functional $J_{\text{ap}}$ admits a unique minimizer $(Z_T, G, W, F)$ in $ H \times \G \times \W \times L^2(0,T; H)$, which enjoys some nice properties given in the lemma below.
\begin{lemma}
	\label{Lem-Min-Of-J-Controls}
	Let $(Z_T, G, W, F)$ denote the unique minimizer of $J_{\text{ap}}$ in $ H \times \G \times \W \times L^2(0,T; H)$. Setting $Z$ the corresponding solution of 
	\begin{equation}
		\label{Z-Corr-to-Min}
		Z' + A^* Z = F, \quad t \in (0,T) \quad \text{ with } \quad Z(T) = Z_T, 
	\end{equation}
	the functions $y$ and $u$ defined by 
	\begin{equation}
		\label{Dictionary-y-u-Z}
		y = F + W + w_*, \quad \hbox{ in } (0,T), \qquad u = B^* Z + G + g_*, \quad \hbox{ in } (0,T),
	\end{equation}
	satisfy the equation \eqref{ControlledEq}, and the conditions \eqref{Moments-Of-U-in-G}, \eqref{ApproxCont-Req}, \eqref{Moments-Of-Y-in-W} and \eqref{Proj-Of-Y-T-in-E}.
\end{lemma}
We then easily deduce Theorem \ref{Thm-ApproxControl} from these two lemmas, whose proofs are done in Section \ref{Subsec-Coercivity} and Section \ref{Subsec-Min-Of-J-Controls}.

\subsection{Proof of Lemma \ref{Lem-Coercivity-of-J}}\label{Subsec-Coercivity}

As $J_{\text{ap}}$ is the sum of convex functions, it is obvious that $J_{\text{ap}}$ will be strictly convex if one of these functions is strictly convex. We claim that the functional $K$ defined by 
$$
	K : (z_T, g, w, f) \in H \times \G \times \W \times L^2(0,T; H)
	\mapsto  
		\int_0^T \norm{ B^* z(t) + g(t) }_{U}^2 \, dt 
		+ 
		\int_0^T \norm{ f(t) + w(t)}_H^2 \, dt,  
$$
where $z$ solves \eqref{Adjoint-NonHom}, is strictly convex. Indeed, according to the unique continuation property \eqref{UC}, $K$ obviously defines a strictly positive quadratic form on $H \times \G \times \W \times L^2(0,T; H)$, so that this is strictly convex, thus entailing the strict convexity of $J_{\text{ap}}$.

Now, to prove the coercivity of $J_{\text{ap}}$, the difficulty is that $K$ does not correspond \emph{in general} to a norm equivalent to $H \times \G \times \W \times L^2(0,T; H)$. Thus, in order to show that $J_{\text{ap}}$ is strictly coercive on $H \times \G \times \W \times L^2(0,T; H)$, we rather proceed by contradiction and take a sequence $(z_{T,n}, g_n, w_n, f_n) \in H \times \G \times \W \times L^2(0,T; H)$ indexed by $n \in \N$ such that
\begin{equation}
	\label{Contradiction-Assumptions-1}
	\lim_{n \to \infty} \rho_n = +\infty, \quad \text{ where } \rho_n =  \| (z_{T,n}, g_n, w_n, f_n)\|_{H \times \G \times \W \times L^2(0,T; H)}, 
\end{equation} 
\begin{equation}
	\label{Contradiction-Assumptions-2}
	\hbox{ and } \quad 
	S = \sup_{n \in \N} J_{\text{ap}} (z_{T,n}, g_n, w_n, f_n) < \infty.
\end{equation}
As usual, we start by renormalizing the data and introduce 
$$
	(\tilde z_{T,n}, \tilde g_n, \tilde w_n, \tilde f_n)
	= 
	\frac{1}{\rho_n} (z_{T,n}, g_n, w_n, f_n), 
$$
so that
\begin{equation}
	\label{Bounds-Tilde-z-g-w-f}
	\forall n \in \N, \quad \| (\tilde z_{T,n}, \tilde g_n, \tilde w_n, \tilde f_n)\|_{H \times \G \times \W \times L^2(0,T; H)} = 1. 
\end{equation}
Using \eqref{Contradiction-Assumptions-2}, we obtain that for all $n \in \N$, 
\begin{multline}
	\label{J-tilde-infini}
	\rho_n^2 
	 \left( 
	 	\frac{1}{2} \int_0^T \| B^* \tilde z_n(t) +\tilde g_n (t) \|_U^2 \, dt
		+
		\frac{1}{2} \int_0^T \| \tilde f_n(t) + \tilde w_n(t) \|_H^2 \, dt
	\right)
	\\
	+ 
	\rho_n
	 \left(	
	 	\langle y_0, \tilde z_n(0) \rangle_H - \langle y_1, \tilde z_{T,n} \rangle_H
		+
		\int_0^T \langle B^* \tilde z_n(t), g_*(t) \rangle \,dt
		+ 
		\int_0^T \langle \tilde f_n(t), w_*(t) \rangle\, dt
		+ 
		\varepsilon \| (I - \P_E) \tilde z_{T,n} \|_H
	 \right) 
	 \\
	 \leq  J_{\text{ap}}(z_{T,n}, g_n, w_n, f_n) \leq S.
\end{multline}
Using \eqref{Contradiction-Assumptions-1} and \eqref{Bounds-Tilde-z-g-w-f}, one easily checks that necessarily, 
\begin{equation}
	\label{Main-Order-infinity}
	\lim_{n \to \infty} 
		\left( 
		\int_0^T \| B^* \tilde z_n +\tilde g_n \|_U^2 \, dt
		+
		\int_0^T \| \tilde f_n + \tilde w_n \|_H^2 \, dt
		\right)
		= 0.
\end{equation}
Now, since $(\tilde z_{T,n}, \tilde g_n, \tilde w_n, \tilde f_n)$ are uniformly bounded in $H \times \G \times \W \times L^2(0,T; H)$ according to \eqref{Bounds-Tilde-z-g-w-f}, and since $\W$ is a  finite-dimensional vector space, there exists $(\tilde z_T, \tilde g, \tilde w, \tilde f) \in H \times \G \times \W \times L^2(0,T; H)$ such that, up to an extraction still denoted the same for simplicity, 
\begin{align}
	\label{Conv-z-T-n}
	&(\tilde z_{T,n})  \underset{n\to\infty}\rightharpoonup \tilde z_T && \hbox{ weakly in } H, 
	\\	
	\label{Conv-g-n}
	& (\tilde g_{n})  \underset{n\to\infty}\rightharpoonup \tilde g && \hbox{ weakly in } L^2(0,T; U),
	\\
	\label{Conv-w-n}
	&(\tilde w_{n})  \underset{n\to\infty}\rightarrow \tilde w && \hbox{ strongly in } L^2(0,T; H),
	\\
	\label{Conv-f-n}
	&(\tilde f_{n})  \underset{n\to\infty}\rightharpoonup \tilde f && \hbox{ weakly in } L^2(0,T; H),
\end{align}
and, from \eqref{Main-Order-infinity} and the above convergences, 
\begin{equation*}
	\int_0^T \| B^* \tilde z +\tilde g \|_U^2 \, dt
		+
	\int_0^T \| \tilde f + \tilde w \|_H^2 \, dt
	= 0, 	
\end{equation*}
where $\tilde z$ is the solution of \eqref{Adjoint-NonHom} with initial datum $\tilde z_T$ and source term $\tilde f$. We thus deduce from the unique continuation property \eqref{UC} that 
\begin{equation*}
	\tilde z_T = 0, \quad \tilde g = 0, \quad \tilde w = 0, \quad \tilde f = 0.
\end{equation*}
The convergences \eqref{Conv-z-T-n}--\eqref{Conv-f-n} then imply that 
\begin{equation*}
	\lim_{n \to \infty} 
		\left( 
		\langle y_0, \tilde z_n(0) \rangle_H - \langle y_1, \tilde z_{T,n} \rangle_H
		+
		\int_0^T \langle B^* \tilde z_n(t), g_*(t) \rangle \,dt
		+ 
		\int_0^T \langle \tilde f_n(t), w_*(t) \rangle\, dt
		\right)
		= 
		0.
\end{equation*}
Therefore, based on \eqref{J-tilde-infini}, we necessarily have 
\begin{equation*}
	\lim_{n \to \infty} \| (I - \P_E) \tilde z_{T,n} \|_H = 0. 
\end{equation*}
Since $E$ is a finite dimensional vector space, with the convergence \eqref{Conv-z-T-n}, we deduce that 
\begin{equation}
	\label{Conv-z-T-n-strong}
	(\tilde z_{T,n})  \underset{n\to\infty}\rightarrow 0 \quad \hbox{ strongly in } H.
\end{equation}
Besides, combining the strong convergence \eqref{Conv-w-n} with \eqref{Main-Order-infinity}, we also have that 
\begin{equation}
	\label{Conv-f-n-strong}
	(\tilde f_{n})  \underset{n\to\infty}\rightarrow 0 \quad \hbox{ strongly in } L^2(0,T;H).
\end{equation}
The strong convergences \eqref{Conv-z-T-n-strong} and \eqref{Conv-f-n-strong} imply that the solution $\tilde z_n$ of $\tilde z_n' + A^* \tilde z_n = \tilde f_n$ in $(0,T)$ with initial datum $\tilde z_n(T) = \tilde z_{T,n}$ strongly converges to $0$ in $L^2(0,T; H)$, so that $B^* \tilde z_n$ strongly converges to $0$ in $L^2(0,T; U)$ and, from \eqref{Main-Order-infinity}, $\tilde g_n$ strongly converges to $0$ in $L^2(0,T; U)$. These strong convergences to $0$ contradict condition \eqref{Bounds-Tilde-z-g-w-f}. This concludes the proof of Lemma \ref{Lem-Coercivity-of-J}.

\begin{remark}
	As noticed in \cite{FabPuelZua}, in fact the above proof shows the following slightly stronger coercivity property: 
	$$
		\liminf_{\| (z_T, g, w, f)\|_{H \times \G \times \W \times L^2(0,T; H)} \to \infty}  
			\frac{J_{\text{ap}}(z_T, g, w, f)}{\| (z_T, g, w, f)\|_{H \times \G \times \W \times L^2(0,T; H)}}
			\geq \varepsilon.
	$$
\end{remark}

%\begin{remark}
%	\label{Rem-W-Inf-Dim-Approx-Proof}
%	%
%	
%	In fact, the condition that $\W$ is of finite dimension could be replaced by the following one: If $w_{n}$ denotes a sequence of elements of $\W$ which is weakly convergent to $0$ in $L^2(0,T; H)$, then the trajectories $z_n$ defined by 
%	%
%	$$
%		z_n' +A^* z_n = w_n, \quad t \in (0,T), \qquad \hbox{ with } z_n(T) = 0
%	$$
%	%
%	strongly converge to $0$ in $L^2(0,T; H)$. 
%	%
%	\\
%	%
%	We refer the interested reader to Remark \ref{Rem-W-Inf-Dim-Approx-Proof} for the explanation of the fact that this condition is sufficient.
%	%
%\end{remark}

%
\subsection{Proof of Lemma \ref{Lem-Min-Of-J-Controls}}\label{Subsec-Min-Of-J-Controls}

Let $(Z_T, G, W, F)$ denote the unique minimizer of $J_{\text{ap}}$ in $ H \times \G \times \W \times L^2(0,T; H)$ and $Z$ the corresponding solution to \eqref{Z-Corr-to-Min}. We will simply write down the Euler-Lagrange equation satisfied by $(Z_T, G, W, F)$, the only difficulty being the possible lack of regularity of the functional $J$ if $\norm{(I - \P_E) Z_T}_H = 0$. 
\medskip

We thus start with the case $\norm{(I - \P_E) Z_T}_H \neq 0$. The functional $J_{\text{ap}}$ is then smooth locally around $(Z_T, G, W, F)$ and the Euler-Lagrange equation satisfied by $(Z_T, G, W, F)$ yields: for all $(z_T, g, w, f) \in H \times \G \times \W \times L^2(0,T; H)$, denoting by $z$ the corresponding solution of \eqref{Adjoint-NonHom}, 
\begin{multline}
	\label{Euler-Lagrange-App}
		0
		= 
		\int_0^T \langle B^* Z(t) + G(t), B^* z(t) + g(t) \rangle_{U} \, dt 
		+ 
		\int_0^T \langle F(t)+ W(t), f(t) + w(t) \rangle_H \, dt
		+
		\langle y_0, z(0) \rangle_H - \langle y_1, z_T \rangle_H
		\\
		+ 
		\int_0^T \langle B^* z(t), g_*(t) \rangle_U \, dt
		+ 
		 \int_0^T \langle f(t), w_*(t) \rangle_H \, dt
		+ 
		\varepsilon \left\langle \frac{(I - \mathbb{P}_E) Z_T}{\norm{ (I - \mathbb{P}_E) Z_T}_H}, z_T \right\rangle_H.
\end{multline}
Note that here, we use the fact that $\mathbb{P}_E$ is the orthogonal projection on $E$ to compute the derivative of the last term.

Now, taking $g = 0$ and $w = 0$ in the above formulation, for all $(z_T, f) \in H \times L^2(0,T; H)$, 
\begin{multline*}
	0 
	= 
	\int_0^T \langle B^* Z(t) + G(t) + g_*(t), B^* z(t) \rangle_{U} \, dt 
	+
	\int_0^T \langle F(t)+ W(t) + w_*(t), f(t) \rangle_H \, dt
	\\
	+
	\langle y_0, z(0) \rangle_H - \langle y_1, z_T \rangle_H
	+
	\varepsilon \left\langle \frac{(I - \mathbb{P}_E) Z_T}{\norm{ (I - \mathbb{P}_E) Z_T}_H}, z_T \right\rangle_H.
\end{multline*}
Now, if we consider $\tilde y$ the solution of \eqref{ControlledEq} with initial datum $y_0$ and control function $u = B^* Z + G + g_*$, which obviously belongs to $L^2(0,T; U)$, and multiply it by solutions $z$ of \eqref{Adjoint-NonHom} with $z_T \in H$ and $f \in L^2(0,T; H)$, we get that 
$$
	0 
	= 
	\int_0^T \langle B^* Z(t) + G(t) + g_*(t), B^* z(t) \rangle_{U} \, dt 
	+
	\int_0^T \langle \tilde y(t), f (t)\rangle_H \, dt
	+
	\langle y_0, z(0) \rangle_H - \langle \tilde y(T), z_T \rangle_H.
$$
Thus, taking $z_T = 0$ and arbitrary $f \in L^2(0,T; H)$, one easily checks that 
$$
	\tilde y = F + W + w_* \quad \hbox{ in } (0,T), 
$$
\emph{i.e.} that $\tilde y$ coincides with $y$ given in \eqref{Dictionary-y-u-Z}. Taking then $f = 0$ and $z_T$ arbitrary in $H$, we deduce that 
$$
	y(T) = y_1 - \varepsilon \frac{(I - \mathbb{P}_E) Z_T}{\norm{ (I - \mathbb{P}_E) Z_T}_H}, 
$$
which of course satisfies
$$
	\norm{y(T) - y_1} \leq \varepsilon \quad \hbox{ and } \quad \P_E y(T) = \P_E y_1. 
$$
We then have to check the properties \eqref{Moments-Of-U-in-G} and \eqref{Moments-Of-Y-in-W}. In order to do that, we simply consider \eqref{Euler-Lagrange-App} in the case $z_T = 0$ and $f= 0$: for all $g \in \G$ and $w \in \W$, 
$$
	0 = 
		\int_0^T \langle B^* Z(t) + G(t),  g(t) \rangle_{U} \, dt 
		+ 
		\int_0^T \langle F(t)+ W(t),  w(t) \rangle_H \, dt.
$$
Consequently $\P_\G (B^* Z + G) = 0$ and $\P_\W (F+ W) = 0$. In view of the definition of $y$ and $u$ in \eqref{Dictionary-y-u-Z}, we immediately deduce \eqref{Moments-Of-U-in-G} and \eqref{Moments-Of-Y-in-W}, thus concluding the proof of Lemma \ref{Lem-Min-Of-J-Controls} when $\norm{(I - \P_E)Z_T}_H \neq 0$.
\medskip

In the case $\norm{(I - \P_E) Z_T}_H = 0$, the functional $J_{\text{ap}}$ is not regular due to the last term in \eqref{Functional-Approx-Cont}, but is still strictly convex.  Using that for all $\eta \in \R$, and $(z_T, g, w, f) \in H \times \G \times \W \times L^2(0,T; H)$, we have 
$$
	J_{\text{ap}}(Z_T, G, W, F) \leq J_{\text{ap}}( (Z_T, G, W, F) + \eta (z_T, g, w, f)),  
$$
we have that 
$$
	\liminf_{\eta \to 0} \left( \frac{J_{\text{ap}}( (Z_T, G, W, F) + \eta (z_T, g, w, f)) - J_{\text{ap}}(Z_T, G, W, F) }{\eta} \right) \geq 0.
$$
Accordingly, we deduce for all $(z_T, g, w, f) \in H \times \G \times \W \times L^2(0,T; H)$, 
\begin{multline*}
		\int_0^T \langle B^* Z(t) + G(t), B^* z(t) + g(t) \rangle_{U} \, dt 
		+ 
		\int_0^T \langle F(t)+ W(t), f(t) + w(t) \rangle_H \, dt
		+
		\langle y_0, z(0) \rangle_H
		\\
		 - \langle y_1, z_T \rangle_H
		+ 
		\int_0^T \langle B^* z(t), g_*(t) \rangle_U \, dt
		+ 
		 \int_0^T \langle f(t), w_*(t) \rangle_H \, dt
		\geq 
		-  \varepsilon \| (I - \P_E) z_T\|_H.
\end{multline*}
Using the similar estimate corresponding to $-(z_T, g, w, f)$, we deduce that for all $(z_T, g, w, f) \in H \times \G \times \W \times L^2(0,T; H)$, 
\begin{multline*}
		\left|
		\int_0^T \langle B^* Z(t) + G(t), B^* z(t) + g(t) \rangle_{U} \, dt 
		+ 
		\int_0^T \langle F(t)+ W(t), f(t) + w(t) \rangle_H \, dt
		+
		\langle y_0, z(0) \rangle_H
		\right.
		\\
		\left.
		 - \langle y_1, z_T \rangle_H
		+ 
		\int_0^T \langle B^* z(t), g_*(t) \rangle_U \, dt
		+ 
		 \int_0^T \langle f(t), w_*(t) \rangle_H \, dt
		\right| 
		\leq 
		\varepsilon \| (I - \P_E) z_T\|_H.
\end{multline*}
The arguments developed above then allow to conclude that $(y,u)$ given by \eqref{Dictionary-y-u-Z} satisfy the equations \eqref{ControlledEq} and that $y(T)$ satisfies, for all $z_T \in H$, 
$$
	\left|
		\langle y(T)- y_1, z_T \rangle_H
	\right|
	\leq \varepsilon \| (I - \P_E) z_T\|_H, 
$$
which implies \eqref{ApproxCont-Req} and \eqref{Proj-Of-Y-T-in-E}. The proofs of \eqref{Moments-Of-U-in-G} and \eqref{Moments-Of-Y-in-W} then follow as before.

\subsection{A remark: relaxing the linear constraint (\ref{Moments-Of-Y-in-W})}

	If we are interested only in a relaxation of the constraints \eqref{Moments-Of-Y-in-W} into 
	\begin{equation}
		\label{Moments-Of-Y-in-W-Rel}
		\norm{ \P_\W y - w_*}_{L^2(0,T; H)} \leq \varepsilon, 
	\end{equation}
	this can be done under the conditions (H1)--(H5) and the unique continuation property \eqref{UC}, even when $\W$ is possibly of infinite dimension. To be more precise, we have the following result:
	
\begin{theorem}[Approximate controllability with linear projection constraints - relaxation of the projection on $\W$]
	\label{Thm-ApproxControl-Bis}
	Let the hypotheses (H1)--(H5) be satisfied, and let $T>0$, and assume the unique continuation property \eqref{UC}.
	\\
	Then for any $y_0$ and $y_1$ in $H$, $\varepsilon >0$, $g_* \in \G$, and $w_* \in \W$, there exists a control function $u \in L^2(0,T; U)$ such that \eqref{Moments-Of-U-in-G} holds, the solution $y$ of \eqref{ControlledEq} satisfies \eqref{ApproxCont-Req}, and the conditions \eqref{Moments-Of-Y-in-W-Rel} and \eqref{Proj-Of-Y-T-in-E}.
\end{theorem}

\begin{proof}[Sketch of the proof]
	The proof of Theorem \ref{Thm-ApproxControl-Bis} simply consists in minimizing the functional 
\begin{multline*}
%	\label{Functional-Approx-Cont-Bis}
		\tilde J_{\text{ap}}(z_T, g,w , f )
		= 
		\frac{1}{2} \int_0^T \norm{ B^* z(t) + g(t) }_{U}^2 \, dt 
		+ 
		\frac{1}{2} \int_0^T \norm{ f(t) + w(t)}_H^2 \, dt
		+
		\langle y_0, z(0) \rangle_H - \langle y_1, z_T \rangle_H
		\\
		+ 
		\int_0^T \langle B^* z(t), g_*(t) \rangle_U \, dt
		+ 
		 \int_0^T \langle f(t), w_*(t) \rangle_H \, dt
		+ 
		\varepsilon \norm{ (I - \mathbb{P}_E) z_T}_H
		+ 
		\varepsilon \norm{ w}_{L^2(0,T; H)} , 
\end{multline*}
defined for 
$$
	(z_T, g,w , f ) \in H \times \G \times \W \times L^2(0,T; H),
$$
where $z$ denotes the solution of \eqref{Adjoint-NonHom}, instead of $J_{\text{ap}}$ in \eqref{Functional-Approx-Cont}.

One can then follow the proof of Theorem \ref{Thm-ApproxControl}, and remark that the only place which uses that $\W$ is of finite dimension is for the proof of the convergence \eqref{Conv-w-n} in the proof of Lemma \ref{Lem-Coercivity-of-J}.

But, in fact, with the addition of the term $\varepsilon \norm{ w}_{L^2(0,T; H)}$ in the functional $\tilde J_{\text{ap}}$, one can prove the coercivity of $\tilde J_{\text{ap}}$ as in Lemma \ref{Lem-Coercivity-of-J}: with the same notations as in the proof of Lemma \ref{Lem-Coercivity-of-J}, one can prove the strong convergence of $\tilde w_n$ to $0$ in $L^2(0,T; H)$ similarly as what is done for $\tilde z_{T,n}$. The detailed proof is left to the reader.

Writing then the optimality conditions for the minimizers, similarly as in Lemma \ref{Lem-Min-Of-J-Controls}, we easily check that the optimum of $\tilde J_{\text{ap}}$ provides a solution to the control problem  \eqref{Moments-Of-U-in-G}, \eqref{ApproxCont-Req}, \eqref{Moments-Of-Y-in-W-Rel} and \eqref{Proj-Of-Y-T-in-E}. 
\end{proof}

%%%%%%%%%%%%%%%%%
%
\section{Exact controllability: Proof of Theorem \ref{Thm-ExactCont}}\label{Sec-ExactCont}
\subsection{Strategy}\label{Subsec-Strategy-ExactCont}
We assume the hypotheses (H1)--(H4), we let $T>0$, and we assume the unique continuation property \eqref{UC}, as well as the observability inequality \eqref{Obs-z-T} for solutions of \eqref{Adjoint-Hom}. 

The first part of the proof of Theorem \ref{Thm-ExactCont} consists in showing the observability inequality \eqref{Obs-z-T-General} for solutions of \eqref{Adjoint-NonHom}, proved in Section \ref{Subsec-Lem-Obs-z-T}: 
\begin{lemma}
	\label{Lem-Obs-z-T-General}
	Within the above setting, there exists a constant $C>0$ such that for all $(z_T, g, w,f) \in H \times \G \times \W \times L^2(0,T; H)$, the solution $z$ of \eqref{Adjoint-NonHom} satisfies \eqref{Obs-z-T-General}.
\end{lemma}
Once this lemma has been obtained, we proceed as in the proof of Theorem \ref{Thm-ApproxControl} with the formal choice $\varepsilon = 0$. To be more precise, we set $(y_0, y_1) \in H^2$, $g_* \in \G$, $w_* \in \W$, and introduce the functional 
\begin{multline}
	\label{Functional-Exact-Cont}
		J_{\text{ex}}(z_T, g,w , f )
		= 
		\frac{1}{2} \int_0^T \norm{ B^* z(t) + g(t) }_{U}^2 \, dt 
		+ 
		\frac{1}{2} \int_0^T \norm{ f(t) + w(t)}_H^2 \, dt
		+
		\langle y_0, z(0) \rangle_H - \langle y_1, z_T \rangle_H
		\\
		+ 
		\int_0^T \langle B^* z(t), g_*(t) \rangle_U \, dt
		+ 
		 \int_0^T \langle f(t), w_*(t) \rangle_H \, dt, 
\end{multline}
defined for 
$$
	(z_T, g,w , f ) \in H \times \G \times \W \times L^2(0,T; H),
$$
where $z$ denotes the solution of \eqref{Adjoint-NonHom}.

The strict convexity of $J_{\text{ex}}$ comes as in the proof of Lemma \ref{Lem-Coercivity-of-J}, while its coercivity immediately follows from the observability property \eqref{Obs-z-T-General} obtained in Lemma \ref{Lem-Obs-z-T-General}. 

We then consider the unique minimizer $(Z_T, G, W, F)$ of $J_{\text{ex}}$ in $H \times \G \times \W \times L^2(0,T; H)$ and proceed as in Lemma \ref{Lem-Min-Of-J-Controls} to deduce a controlled trajectory $y$ and a control function $u$ which satisfy all the requirements (in fact, it is even easier here as the functional $J_{\text{ex}}$ is differentiable everywhere in $H \times \G \times \W \times L^2(0,T; H)$). Details of the proof are left to the reader.

\subsection{Proof of Lemma \ref{Lem-Obs-z-T-General}}\label{Subsec-Lem-Obs-z-T}

For $z_T \in H$, we introduce $\tilde z$ as the solution of $\tilde z' + A^* \tilde z = 0$ in $(0,T)$ and $\tilde z(T) = z_T$.

From the observability property \eqref{Obs-z-T}, we thus get a constant $C>0$ such that for all $z_T \in H$, 
$$
	\norm{z_T}_H\leq C \norm{B^* \tilde z}_{L^2(0,T; U)}. 
$$
We then use that $B^* \in \mathscr{L}(H,U)$ and that there exists $C>0$ such that for all $f \in L^2(0,T; H)$, the solution $z_f' + A^* z_f = f$ in $(0,T)$ and $ z_f(T) = 0$ satisfies
$$
	\norm{z_f}_{L^2(0,T; H)} \leq C \norm{f}_{L^2(0,T; H)}. 
$$
Therefore, we easily get a constant $C>0$ such that for all $z_T \in H$ and $f \in L^2(0,T; H)$, the solution $z$ of \eqref{Adjoint-NonHom} satisfies
$$
	\norm{z_T}_H\leq C \norm{B^* z}_{L^2(0,T; U)} + C \norm{f}_{L^2(0,T; H)}. 
$$
Indeed, this can be easily deduced by writing $ \tilde z = z -z_f$ and using the above estimates.
\\
We then deduce the existence of a constant $C>0$ such that for all $(z_T, g,w,f)\in H \times \G \times \W \times L^2(0,T; H)$, 
\begin{multline}
	\label{Obs-z-T-Preliminary}
	\| (z_T, g,w,f)\|_{H \times \G \times \W \times L^2(0,T; H)}
	\\
	\leq 
	C\left( \norm{B^* z}_{L^2(0,T; U)} +  \norm{f}_{L^2(0,T; H)} + \norm{g}_{L^2(0,T; U)} +  \norm{w}_{L^2(0,T; H)} \right).
\end{multline}

Now, we can deduce the observability inequality \eqref{Obs-z-T-General} by contradiction. Assume that we have a sequence $(z_{T,n}, g_n,w_n,f_n)\in H \times \G \times \W \times L^2(0,T; H)$ such that 
\begin{align}
	\label{ContradictionAss1-Exact}	
	& \forall n \in \N, \quad \norm{(z_{T,n}, g_n,w_n,f_n)}_{ H \times \G \times \W \times L^2(0,T; H)} = 1, 
	\\
	\label{ContradictionAss2-Exact}	
	&
	\hbox{ and } 
	\quad
	\lim_{n \to \infty} \left(\norm{ B^* z_n + g_n}_{L^2(0,T; U)} + \norm{ f_n + w_n}_{L^2(0,T; H)}\right) = 0.
\end{align}
From \eqref{ContradictionAss1-Exact} and the fact that $\G$ and $\W$ are of finite dimension, we obtain the following convergences: there exists $(z_{T}, g, w, f) \in H \times \G \times \W \times L^2(0,T; H)$ such that, up to an extraction still denoted the same for simplicity,
\begin{align}
	\label{Conv-z-T-n-ex}
	&(z_{T,n})  \underset{n\to\infty}\rightharpoonup  z_T && \hbox{ weakly in } H, 
	\\	
	\label{Conv-g-n-ex}
	& (g_{n})  \underset{n\to\infty}\rightarrow  g && \hbox{ strongly in } L^2(0,T; U),
	\\
	\label{Conv-w-n-ex}
	&(w_{n})  \underset{n\to\infty}\rightarrow  w && \hbox{ strongly in } L^2(0,T; H),
	\\
	\label{Conv-f-n-ex}
	&(f_{n})  \underset{n\to\infty}\rightharpoonup  f && \hbox{ weakly in } L^2(0,T; H),
\end{align}
and,  from \eqref{ContradictionAss2-Exact} and the above convergences, 
\begin{equation*}
	\int_0^T \| B^*  z + g \|_U^2 \, dt
		+
	\int_0^T \|  f +  w \|_H^2 \, dt
	= 0, 	
\end{equation*}
where $ z$ is the solution of \eqref{Adjoint-NonHom}. It follows from \eqref{UC} that $z_T = 0$, $ g = 0$, $w = 0$ and $f = 0$. Thus, in view of the strong convergences \eqref{Conv-g-n-ex}--\eqref{Conv-w-n-ex}, the condition \eqref{ContradictionAss2-Exact} implies that $B^* z_n$ strongly converges to $0$ in $L^2(0,T; U)$ and $f_n$ strongly converges to $0$ in $L^2(0,T; H)$, making the condition \eqref{ContradictionAss1-Exact} incompatible with the observability estimate \eqref{Obs-z-T-Preliminary}. This completes the proof of Lemma \ref{Lem-Obs-z-T-General}.
%

%%%%%%%%%%%%%%%%%
%
\section{Null controllability: Proofs of Theorem \ref{Thm-NullCont} and Proposition~\ref{Prop-From-Usual-Obs-z0-to-Obs-z0-Gal}}
\label{Sec-NullCont}

\subsection{Proof of Theorem \ref{Thm-NullCont}}
\label{Subsec-Proof-NullCont}
We assume the hypotheses (H1)--(H4), we let $T>0$, and we assume the observability inequality \eqref{Obs-z-0-General} for solutions of \eqref{Adjoint-NonHom}. 

We then take $y_0 \in H$, $g_* \in \G$, $w_* \in \W$, and introduce the functional 
\begin{multline}
	\label{Functional-Null-Cont}
		J_{\text{nu}}(z_T, g,w , f )
		= 
		\frac{1}{2} \int_0^T \norm{ B^* z(t) + g(t) }_{U}^2 \, dt 
		+ 
		\frac{1}{2} \int_0^T \norm{ f(t) + w(t)}_H^2 \, dt
		+
		\langle y_0, z(0) \rangle_H 
		\\
		+ 
		\int_0^T \langle B^* z(t), g_*(t) \rangle_U \, dt
		+ 
		 \int_0^T \langle f(t), w_*(t) \rangle_H \, dt, 
\end{multline}
defined for 
$$
	(z_T, g,w , f ) \in H \times \G \times \W \times L^2(0,T; H),
$$
where $z$ denotes the solution of \eqref{Adjoint-NonHom}.

We then introduce the set 
\begin{multline}
	\label{Def-Set-N}
	N = \{(z_T, g, w, f) \in H \times \G \times \W \times L^2(0,T; H),
	\\
	 \hbox{ for which } B^* z + g = 0 \hbox{ in } L^2(0,T; U),  \hbox{ and } f + w = 0 \hbox{ in } L^2(0,T; H) \}.
\end{multline}	
According to the observability inequality \eqref{Obs-z-0-General} for solutions of \eqref{Adjoint-NonHom}, this set $N$ can be characterized as follows: 
\begin{equation}
	\label{Set-N}
	N = \{ (z_T, 0, 0, 0) \hbox{ such that  the solution $z$ of \eqref{Adjoint-Hom}  satisfies } B^* z = 0 \hbox{ in } (0,T) \hbox{ and } z(0) = 0\}.
\end{equation}
In many situations, the set $N$ in \eqref{Set-N} is reduced to $\{0\}$. This is in particular the case when $A$ generates an analytic semi-group, see \cite[Remark 17]{TrelatWangXu}. However, in the general setting we are dealing with in Theorem \ref{Thm-NullCont}, it is not clear that we can guarantee that $N$ is reduced to $\{0\}$.

In view of the observability inequality, it is natural to introduce the space
\begin{equation}
	\label{Def-X-0}
	X_0 = ( H \times \G \times \W \times L^2(0,T; H)) /\!\raisebox{-.65ex}{\ensuremath{N}},
\end{equation}
and to endow it with the norm
\begin{equation*}
	\norm{(z_T, g,w , f )}_{obs}^2
	=
	\int_0^T \norm{ B^* z(t) + g(t) }_{U}^2 \, dt 
		+ 
	 \int_0^T \norm{ f(t) + w(t)}_H^2 \, dt, 
\end{equation*}
where $z$ solves \eqref{Adjoint-NonHom}. One then easily checks that $J_{\text{nu}}$ is well-defined on $X_0$ and coercive for this norm. 

However, it is in general not true that this norm corresponds to the $H \times \G \times \W \times L^2(0,T; H)$ topology. We should thus define
\begin{equation}
	\label{X-Obs}
	X_{obs} = \overline{X_0}^{\|\cdot \|_{obs}}, 
\end{equation}
\emph{i.e.} the completion of $X_0$ in \eqref{Def-X-0} for the topology induced by the norm $\norm{\cdot }_{obs}$. 

Then, according to the observability estimate \eqref{Obs-z-0-General}, the functional $J_{\text{nu}}$ is continuous for the topology induced by $\norm{\cdot }_{obs}$. It can thus be extended by continuity to the space $X_{obs}$, and we will denote this extension by $J_{\text{nu}}$ as well with a slight abuse of notations. Besides, the observability estimate \eqref{Obs-z-0-General} also yields that the functional $J_{\text{nu}}$ is also coercive and strictly convex in $X_{obs}$. It thus admits a unique minimizer $(Z_T,G, W, F)$ in $X_{obs}$. 

Writing the corresponding Euler-Lagrange equations, we can proceed as in the proof of Lemma \ref{Lem-Min-Of-J-Controls} and deduce that, taking the control $u$ and the trajectory $y$ as in \eqref{Dictionary-y-u-Z}, we can solve the null-controllability problem \eqref{NullCont-Req} with the constraints \eqref{Moments-Of-U-in-G}, \eqref{Moments-Of-Y-in-W} on the projection of the control and of the trajectory. The only point to check is that $u$ indeed belongs to $L^2(0,T; U)$ as claimed. This is due to the fact that the minimizer $(Z_T, G, W, F)$ belongs to $X_{obs}$, so that $B^*Z + G$ is well defined as an element of $L^2(0,T; U)$. Details are left to the reader.

\subsection{Proof of Proposition \ref{Prop-From-Usual-Obs-z0-to-Obs-z0-Gal}}
\label{Subsec-Proof-Prop}
%%%%%%%%%%%%%%%%%
%
We place ourselves in the setting of Proposition \ref{Prop-From-Usual-Obs-z0-to-Obs-z0-Gal}. The proof of Proposition \ref{Prop-From-Usual-Obs-z0-to-Obs-z0-Gal} is rather close to the one of Lemma \ref{Lem-Obs-z-T-General}.

First, we start by remarking that one can immediately deduce from the observability inequality \eqref{Obs-z-tilde-T} for solutions of \eqref{Adjoint-Hom} that there exists a constant $C>0$ such that for all $z_T \in H$ and $f \in L^2(0,T; H)$, the solution of \eqref{Adjoint-NonHom} satisfies
\begin{equation}
	\label{Obs-z-tilde-T-Pre-0}
	\norm{ z(\tilde T)}_H \leq C \left(\norm{B^* z}_{L^2(0,T;U)} + \norm{f}_{L^2(0,T; H)} \right). 
\end{equation}

Then, to prove the observability inequality \eqref{Obs-z-0-General}, we use a contradiction argument. Namely we consider a sequence $(z_{T,n}, g_n, w_n, f_n) \in H \times \G \times \W \times L^2(0,T; H)$ such that

\begin{align}
	\label{ContradictionAss1-Null}	
	& \forall n \in \N, \quad \norm{(z_{n}(0), g_n,w_n,f_n)}_{ H \times \G \times \W \times L^2(0,T; H)} = 1, 
	\\
	\label{ContradictionAss2-Null}	
	&
	\hbox{ and } 
	\quad
	\lim_{n \to \infty} \left(\norm{ B^* z_n + g_n}_{L^2(0,T; U)} + \norm{ f_n + w_n}_{L^2(0,T; H)}\right) = 0.
\end{align}
From \eqref{ContradictionAss1-Null} and \eqref{ContradictionAss2-Null}, we deduce from \eqref{Obs-z-tilde-T-Pre-0} that $\| z_n(\tilde T) \|_H$ is uniformly bounded. 

Using thus \eqref{ContradictionAss1-Null} and \eqref{ContradictionAss2-Null} and the fact that $\G$ and $\W$ are of finite dimension, we obtain the following convergences: there exists $(z_{\tilde T}, g, w, f) \in H \times \G \times \W \times L^2(0,T; H)$ such that, up to an extraction still denoted the same, 
\begin{align}
	\label{Conv-z-T-n-null}
	&(z_{n}(\tilde T))  \underset{n\to\infty}\rightharpoonup  z_{\tilde T} && \hbox{ weakly in } H, 
	\\	
	\label{Conv-g-n-null}
	& (g_{n})  \underset{n\to\infty}\rightarrow  g && \hbox{ strongly in } L^2(0,T; U),
	\\
	\label{Conv-w-n-null}
	&(w_{n})  \underset{n\to\infty}\rightarrow  w && \hbox{ strongly in } L^2(0,T; H),
	\\
	\label{Conv-f-n-null}
	&(f_{n})  \underset{n\to\infty}\rightharpoonup  f && \hbox{ weakly in } L^2(0,T; H),
\end{align}
and,  from \eqref{ContradictionAss2-Null} and the above convergences, 
\begin{equation}
	\label{limit-tilde-t-t}
	\int_0^{\tilde T} \| B^*  z + g \|_U^2 \, dt
		+
	\int_0^T \|  f +  w \|_H^2 \, dt
	= 0, 	
\end{equation}
where $z$ is the solution of
\begin{equation*}
	\left\{ 
		\begin{array}{l}
		 z' + A^* z = f, \quad t \in (0,\tilde T),
		\\
		z(\tilde T) = z_{\tilde T}. 
		\end{array}
	\right.
\end{equation*}
Note that equation \eqref{limit-tilde-t-t} involves two integrals, one in $(0,\tilde T)$, the other one in $(0,T)$. This is due to the fact that we do not have any argument to pass to the limit in $z_n$ in the interval $(\tilde T, T)$. 

We can then use the unique continuation property \eqref{UC-Tilde} to get that $z_{\tilde T} = 0$, $g = 0$, $w = 0$ and $f = 0$. Besides, from \eqref{ContradictionAss2-Null} and the strong convergences \eqref{Conv-g-n-null}--\eqref{Conv-w-n-null} (recall that $g = 0$ and $w = 0$), we get that 
$$
	\lim_{n \to \infty} \left( \|B^* z_n\|_{L^2(0,T; U)} + \| f_n \|_{L^2(0,T; H)} \right) = 0, 
$$
so that from \eqref{Obs-z-tilde-T-Pre-0}, we obtain that 
\begin{equation*}
	\lim_{n \to \infty} \|z_n(\tilde T)\|_H = 0. 
\end{equation*}
From the two above convergences, we easily deduce that 
\begin{equation*}
	\lim_{n \to \infty} \left( \|z_n(0)\|_H + \| f_n \|_{L^2(0,T; H)} \right) = 0.
\end{equation*}
With the strong convergences \eqref{Conv-g-n-null}--\eqref{Conv-w-n-null} (recall that $g = 0$ and $w = 0$), this contradicts the assumption \eqref{ContradictionAss1-Null}. This completes the proof of Proposition \ref{Prop-From-Usual-Obs-z0-to-Obs-z0-Gal}.

\section{Examples}\label{Sec-Examples}
The goal of this section is to derive several instances in which the unique continuation property \eqref{UC} can be proved. We do not aim at giving the most general setting for each instance, but rather at describing possible strategies to prove the unique continuation property \eqref{UC} for some specific spaces $\W$ and $\G$.

\subsection{Example 1. $\W = \{0\}$ and $\G$ containing functions supported on $(0,T')$ with $T' \in (0,T)$.} 

Example $1$ focuses on the following case. We let $T>0$, $T' \in (0,T)$, and we choose
\begin{equation}
\label{Example-1}
	\G = \{g \in L^2(0,T; U), \, g = 0 \hbox{ in } (T', T)\}, 
	\qquad 
	\W = \{0 \},
\end{equation}
where $T' \in (0,T)$ is such that 
\begin{equation}
	\label{UC-classical-T'}
	\hbox{If } z \hbox{ satisfies }
	\left\{ 
		\begin{array}{ll}
		 z' + A^* z = 0, \quad & \hbox{ for } t \in (T', T),
		\\
		z(T) = z_T \in H, 
		\\
		B^* z = 0,  \quad & \hbox{ for } t \in (T',T),
		\end{array} 
	\right.	
	\qquad \hbox{ then } \qquad 
		z_T = 0.
\end{equation}
Then condition \eqref{UC} is easy to check, so that Theorem \ref{Thm-ApproxControl}  applies. 

Of course, this case is somehow a straightforward case, as the unique continuation property \eqref{UC-classical-T'} implies that given any $y_{T'}, y_1 \in H^2$, and any $\varepsilon >0$, there exists a control function $u \in L^2(T', T; U)$ such that the solution $y$ of 
\begin{equation*}
	y' = A y + B u, \quad t \in (T', T), \qquad \hbox{ with } y(T') = y_{T'}
\end{equation*}
satisfies \eqref{ApproxCont-Req}. 

Thus, if one wants to approximately control \eqref{ControlledEq} and to impose the condition $\P_\G u = g_*$ for some $g_* \in \G$, the simplest thing to do is to take $u = g_*$ in $(0,T')$, call $y_{T'}$ the state obtained by solving the equation \eqref{ControlledEq} on $(0,T')$ starting from $y_0$, and then use the above approximate controllability property to conclude the argument.
\medskip

More interesting results arise when considering exact controllability result with a subspace $\G$ of finite dimension when the time of unique continuation (equivalently of approximate controllability) and the time of observability (equivalently of exact controllability) do not coincide. 
\medskip

Such an instance is given by the wave equation. Let $\Omega$ be a smooth bounded domain of $\R^d$ and $\omega$ a non-empty subdomain of $\Omega$, and consider the controlled wave equation:
\begin{equation}
	\label{WaveEq}
		\left\{
			\begin{array}{ll}
				\partial_{tt} y - \Delta y = u \chi_\omega, \quad &  \hbox{ for } (t,x) \in (0,T) \times \Omega, 
				\\
				y(t,x) = 0, \quad & \hbox{ for } (t,x) \in (0,T) \times \partial \Omega, 
				\\
				(y(0, \cdot), \partial_t y(0, \cdot)) = (y_0, y_1), \quad &  \hbox{ for } x \in \Omega.
			\end{array}
		\right.
\end{equation}
Here, the state is given by $Y = (y, \partial_t y)$ and $u$ is the control function. The function $\chi_\omega$ is the indicator function of the set $\omega$. This system writes under the form \eqref{ControlledEq} with 
\begin{equation*}
	Y = \left( \begin{array}{c} y \\ \partial_t y \end{array} \right), 
	\quad
	A = \left( \begin{array}{cc} 0 & Id \\ \Delta & 0 \end{array} \right), 
	\quad 
	B = \left( \begin{array}{c} 0 \\ \chi_\omega \end{array} \right), 
\end{equation*}
with $H = H^1_0(\Omega) \times L^2(\Omega)$, $\mathscr{D}(A) = H^2\cap H^1_0(\Omega) \times H^1_0(\Omega)$, $U = L^2(\omega)$.
\\
The following facts are well-known:
\begin{itemize}
	\item The wave equation \eqref{WaveEq} is approximately controllable if 
	$$
		T > T_{AC} : = 2 \sup_{\Omega} d(x,\omega), 
	$$
	where $d$ denotes the geodesic distance in $\Omega$. This is a consequence of Holmgren's uniqueness theorem, see \cite{Russell-1971,Russell-1971-II}. 
	\item The wave equation \eqref{WaveEq} is exactly controllable in time $T$ if (and only if) the so-called geometric control condition is satisfied, see \cite{Bardos,Bardos-L-R,BurqGerard}. Roughly speaking, this consists in saying that all the rays of geometric optics meet the domain $\omega$ before the time $T$. We call $T_{EC}$ the critical time for exact controllability.
\end{itemize}
Of course, the critical time of exact controllability is always larger than the one of approximate controllability, but our results will be of interest only when $T_{EC}$ is strictly larger than $T_{AC}$, that we assume from now on. This happens, for instance, in the case of the wave equation in the $2$-d unit square $\Omega = (0,1)^2$ observed from a $\delta$-neighborhood of two consecutive sides, in which $T_{AC} = 2(1- \delta)$ and $T_{EC} = 2 \sqrt{2(1- \delta)}$. 
\\
We then choose:
\begin{multline}
	\label{Setting-T-ec-T-ac}
	T > T_{EC}, \quad T' \in (0,T) \hbox{ such that } T- T' > T_{AC}, 
	\\
	\G\,  \hbox{ a finite dimensional subspace of } \{g \in L^2(0,T; L^2(\omega)), \, g = 0 \hbox{ in } (T', T)\times \omega\}, 
	\quad
	\W = \{0\}.
\end{multline}
We then claim the following result: 
\begin{theorem}
	\label{Thm-Wave}
	Let us consider the setting of \eqref{Setting-T-ec-T-ac}. Then for all $(y_0, y_1) \in H^1_0(\Omega) \times L^2(\Omega)$ and $(y_0^T, y_1^T)\in H^1_0(\Omega) \times L^2(\Omega)$, for all $g \in \G$, there exists a control function $u \in L^2(0,T; L^2(\omega))$ such that the solution $y$ of \eqref{WaveEq} starting from $(y_0, y_1)$ satisfies $(y(T), \partial_t y(T)) = (y_0^T, y_1^T)$ in $\Omega$ and $\P_\G u =  g$. 
\end{theorem}

This result is an easy consequence of Theorem \ref{Thm-ExactCont} in the setting corresponding to the wave equation. We shall thus check the main two assumptions in Theorem \ref{Thm-ExactCont}. The first one is that the unique continuation property \eqref{UC-classical-T'} holds for the adjoint of the wave equation, which is given by 
\begin{equation}
	\label{WaveEq-Adj}
		\left\{
			\begin{array}{ll}
				\partial_{tt} z - \Delta z = 0 , \quad &  \hbox{ for }(t,x) \in (0,T) \times \Omega, 
				\\
				z(t,x) = 0, \quad &  \hbox{ for }(t,x) \in (0,T) \times \partial \Omega, 
				\\
				(z(T, \cdot), \partial_t z(T, \cdot)) = (z_0, z_1), &  \hbox{ for }x \in \Omega.
			\end{array}
		\right.
\end{equation}
The usual choice when dealing with the wave equation is to identify $L^2(\Omega)$ with its dual, see e.g. \cite{Lions}. Therefore, if the controlled trajectory takes value in $H^1_0(\Omega) \times L^2(\Omega)$, the solutions of the adjoint equation have to be considered as taking values in $L^2(\Omega) \times H^{-1}(\Omega)$. This slight shift of spaces does not modify at all the validity of the abstract arguments developed in Theorem \ref{Thm-ExactCont}.

Holmgren's uniqueness theorem \cite{John-1949,Hormander-I} gives that if a solution $z$ of \eqref{WaveEq-Adj} with initial datum in $L^2(\Omega) \times H^{-1}(\Omega)$ satisfies $z (t,x) = 0$ for $(t,x) \in (T',T) \times \omega$, then $z$ vanishes identically on $(T',T) \times \Omega$. Therefore, the unique continuation property \eqref{UC} is satisfied here.
\\
The second assumption to check is the following observability inequality: there exists $C>0$ such that for all solutions $z$ of \eqref{WaveEq-Adj} with $(z_0, z_1) \in L^2(\Omega) \times H^{-1}(\Omega)$, 
\begin{equation}
	\label{Obs-Waves}
	\norm{ (z_0, z_1)}_{L^2(\Omega) \times H^{-1}(\Omega)} \leq C \norm{ z \chi_\omega}_{L^2((0,T) \times \omega)}. 
\end{equation}
This inequality is equivalent to the exact controllability of the wave equation (see \cite{Lions}), and is thus satisfied when the wave equation is exactly controllable, \emph{i.e.} when $T > T_{EC}$ (In fact, the proof of exact controllability of the wave equation in \cite{Bardos,Bardos-L-R} relies on the proof of \eqref{Obs-Waves}).

On this example, it clearly appears that $\G$ has to be of finite dimension. Indeed, Theorem \ref{Thm-Wave} could not be true when taking  $\G =  \{g \in L^2(0,T; L^2(\omega)), \, g = 0 \hbox{ in } (T', T)\}$. Otherwise, imposing $\P_\G u = 0$, one should be able to steer solutions of \eqref{WaveEq} from any initial datum to any final datum in $H^1_0(\Omega) \times L^2(\Omega)$ with controls vanishing on $(0,T')$. It is then not difficult to check that this would imply exact controllability of the wave equation in time $T - T'$, which could be taken smaller than $T_{EC}$ if $T_{AC} < T_{EC}$.

\subsection{Example 2. The case $\G = \{0\}$, when $B^*z  = 0 $ implies that $w = 0$, and related examples}\label{Subsec-Ex2}

To start with, we consider the case $\G = \{ 0 \}$.

To motivate the functional setting given afterwards, we start with the example of the heat equation in $\Omega$ with a distributed control supported in a non-empty open subset $\omega$, namely
\begin{equation}
	\label{HeatEq}
		\left\{
			\begin{array}{ll}
				\partial_{t} y - \Delta y = u \chi_\omega, \quad & \hbox{ for } (t,x) \in (0,T) \times \Omega, 
				\\
				y(t,x) = 0, \quad & \hbox{ for } (t,x) \in (0,T) \times \partial \Omega, 
				\\
				y(0, \cdot) = y_0, & \hbox{in } \Omega,
			\end{array}
		\right.
\end{equation}
and we assume that the vector space $\W$ is a vector space of $L^2(0,T; L^2(\Omega))$ such that
\begin{equation}
	\label{Cond-W-Ex2}
	\Pi_{\omega} : f \mapsto f|_{\omega} \hbox{ satisfies } \hbox{Ker\,}(\Pi_{\omega}|_{\W} ) = \{0\}.
\end{equation}   
Here, we did not make precise the space on which $\Pi_{\omega}$ is defined, but in view of our applications below, it is natural to define it as going from $L^2(0,T; L^2(\Omega))$ to $L^2(0,T; L^2(\omega))$. Note that this example is borrowed from \cite{Nakoulima-2004,Mophou-Nakoulima-2008,Mophou-Nakoulima-2009} and is also closely related to the condition (6.5) in \cite{Lions-Sentinelles-1992}. 

To make it fit into the abstract setting of \eqref{ControlledEq}, it suffices to take 
\begin{equation*}
	A = \Delta, \ \hbox{ in } H = L^2(\Omega), \hbox{ with } \mathscr{D}(A) = H^2 \cap H^1_0(\Omega), 
	\quad 
	\hbox{ and } 
	\quad
	B = \chi_\omega, \ \hbox{ with } U = L^2(\omega).
\end{equation*}

The corresponding unique continuation property \eqref{UC} writes as follows: if $z$ satisfies
\begin{equation}
	\label{HeatEq-Adj-Ex2}
		\left\{
			\begin{array}{ll}
				\partial_{t} z + \Delta z = w, \quad &  \hbox{ for } (t,x) \in (0,T) \times \Omega, 
				\\
				z(t,x) = 0, \quad &  \hbox{ for }(t,x) \in (0,T) \times \partial \Omega, 
				\\
				z(T, \cdot) = z_T, & \hbox{in } \Omega, 
			\end{array}
		\right.
\end{equation}
with $z_T \in L^2(\Omega)$ and $w \in \W$, and 
\begin{equation}
	\label{UC-heat-Ex2}
	z(t,x) = 0 \quad \hbox{ in } (0,T) \times \omega, 
\end{equation}
then one should have $w = 0$ and $z_T = 0$. This is indeed the case provided \eqref{Cond-W-Ex2} holds. Indeed, \eqref{UC-heat-Ex2} implies that $\partial_t z + \Delta z = 0$ in $H^{-2}((0,T) \times \omega)$, so that $\Pi_{\omega}(w) = 0$, and thus from \eqref{Cond-W-Ex2}, $w = 0$. One can then deduce that $z = 0$ in $(0,T) \times \Omega$ from classical unique continuation properties for the heat equation. 

We can then use that the heat equation is null-controllable in any time $T>0$, see \cite{FursikovImanuvilov,LebRob}, or equivalently final state observable at any time, meaning that for all $T>0$, there exists $C$ such that any solution $z$ of \eqref{HeatEq-Adj-Ex2} with $z_T \in L^2(\Omega)$ and $w = 0$ satisfies 
$$
	\norm{z(0)}_{L^2(\Omega)} \leq C \norm{z \chi_\omega}_{L^2((0,T) \times \omega)}.
$$
In particular, for any time $\tilde T \in (0,T)$, there exists a constant $C(\tilde T)$ such that any solution $z$ of \eqref{HeatEq-Adj-Ex2} with $z_T \in L^2(\Omega)$ and $w = 0$ satisfies 
$$
	\norm{z(\tilde T)}_{L^2(\Omega)} \leq C(\tilde T) \norm{z \chi_\omega}_{L^2((0,T) \times \omega)}.
$$

Therefore, combining Proposition \ref{Prop-From-Usual-Obs-z0-to-Obs-z0-Gal} and Theorem \ref{Thm-NullCont}, we get that if $\W$ is of finite dimension and satisfies condition \eqref{Cond-W-Ex2}, then for any $y_0 \in L^2(\Omega)$ and $w \in \W$, there exists a control function $u \in L^2(0,T; L^2(\omega))$ such that the controlled trajectory $y$ of \eqref{HeatEq} starting from $y_0$ satisfies $y(T) = 0$ and $\P_\W y= w$. (The existence of $T' < T$ such that $\Pi_{\omega, T'} : f \mapsto  f|_{\omega \times (0,T')}$ satisfies  $\hbox{Ker\,}(\Pi_{\omega,T'}|_{\W} ) = \{0\}$ can be proved easily by contradiction using the fact that $\hbox{Ker\,}(\Pi_{\omega,T}|_{\W} ) = \{0\}$ and $\W$ is of finite dimension.)
\medskip

In fact, the above example can be put into a much more abstract setting, by assuming that the space $\W$ is such that 

\begin{equation}
	\label{Cond-Mophou}
	\exists \text{ two linear operators } K \text{ and } L \text{ s.t. }
	\left\{
	\begin{array}{l}
		\ds K: L^2(0,T; H) \mapsto \mathcal{H} \hbox{ for some Hilbert space } \mathcal{H}, 
		\\
		\ds L : L^2(0,T; U) \mapsto \mathcal{H}, 
		\\
		\ds K (\partial_t + A^* )  = L B^* ,
		\\
		\ds \hbox{Ker\,} (K|_{\W}) = \{0\}.
	\end{array}
	\right.
\end{equation}
The above example fits into this setting with $K$ being the restriction operator to $(0,T) \times \omega$ and $L =  (\partial_t + \Delta)$ and $\mathcal{H} = H^{-2}((0,T)\times\omega)$.

Now, if $\W$ satisfies \eqref{Cond-Mophou} and for some $w \in \W$ and $z \in C^0([0,T]; H)$, 
$$
	(\partial_t + A^* )z = w \quad \hbox{ and } \quad B^* z = 0, \quad \hbox{ in } (0,T), 
$$
then 
$$
	K w = K ((\partial_t +A^*) z) = L B^* z = 0, 
$$
so that $w = 0$ according to the condition $ \hbox{Ker\,} (K|_{\W}) = \{0\}$. In particular, the classical unique continuation property \eqref{UC-Classical} would then imply the more evolved unique continuation property \eqref{UC}.

Of course, there is a completely symmetric statement when considering $\W = \{0\}$ and $\G$ such that 
\begin{equation}
	\label{Cond-Mophou-bis}
	\exists \text{ two linear operators } K \text{ and } L \text{ s.t. }
	\left\{
	\begin{array}{l}
		\ds K: L^2(0,T; H) \mapsto \mathcal{H} \hbox{ for some Hilbert space } \mathcal{H}, 
		\\
		\ds L : L^2(0,T; U) \mapsto \mathcal{H}, 
		\\
		\ds K (\partial_t + A^* )  = L B^* ,
		\\
		\ds \hbox{Ker\,} (L|_{\G}) = \{0\}.
	\end{array}
	\right.
\end{equation}
In the above example \eqref{HeatEq} for instance, when $\G = \hbox{Span\,} \{g \}$ for some $g \in L^2(0,T; L^2(\omega))$ such that $\norm{(\partial_t + \Delta) g}_{H^{-2}((0,T)\times \omega)} \neq 0$, then the above condition is satisfied with $K: L^2( 0,T; L^2( \Omega)) \to H^{-2}((0,T) \times \omega)$ the usual restriction operator and $L = \partial_t + \Delta : L^2(0,T; L^2(\omega)) \to H^{-2}((0,T) \times \omega)$.

In fact, under condition \eqref{Cond-Mophou} or \eqref{Cond-Mophou-bis}, it is not even needed to assume that one of the two vector spaces $\W$ or $\G$ is reduced to $\{0\}$ if we have that $(w, g) \in \W \times \G \mapsto K w + L g$ is injective.

\subsection{Example 3. Using time differentiation to go back to a classical unique continuation property}\label{Subsec-Ex3}

In order to motivate the introduction of our abstract setting, let us consider again the heat equation \eqref{HeatEq} in the case $\W = \hbox{Span\,} \{ e^{\mu t} w_\mu(x) \}$ for some $\mu \in \R$ and $w_\mu \in L^2(\Omega)$ (In fact, this example is inspired by the previous work \cite[Theorem 4.4 and Lemma 4.6]{Chowdhury-Erv-2019}). Then to prove the unique continuation property \eqref{UC}, we want to prove that if $z$ satisfies
\begin{equation}
	\label{HeatEq-Adj-Ex3}
		\left\{
			\begin{array}{ll}
				\partial_{t} z + \Delta z = a_\mu e^{\mu t} w_\mu, \quad & \hbox{ for } (t,x) \in (0,T) \times \Omega, 
				\\
				z(t,x) = 0, \quad &  \hbox{ for }(t,x) \in (0,T) \times \partial \Omega, 
				\\
				z(T, \cdot) = z_T, & \hbox{in } \Omega, 
			\end{array}
		\right.
\end{equation}
with $z_T \in L^2(\Omega)$ and $a_\mu \in \R$, and 
\begin{equation*}
%	\label{UC-heat-Ex3}
	z(t,x) = 0 \quad \hbox{ in } (0,T) \times \omega, 
\end{equation*}
then $a_\mu = 0$ and $z$ vanishes everywhere in $(0,T) \times \Omega$.

In order to solve this problem, the basic idea is again to ``kill'' the source term by applying a suitable operator. Here, in view of the time dependence of the source term, it is natural to apply $\partial_t - \mu$ to the equation \eqref{HeatEq-Adj-Ex3}. In particular, if we set $\tilde z = (\partial_t - \mu) z$, we obtain that 
\begin{equation*}
%	\label{Eq-Tilde-Z}
		\left\{
			\begin{array}{ll}
				\partial_{t} \tilde z + \Delta  \tilde z = 0, \quad &  \hbox{ for }(t,x) \in (0,T) \times \Omega, 
				\\
				\tilde z(t,x) = 0, \quad &  \hbox{ for }(t,x) \in (0,T) \times \partial \Omega,
			\end{array}
		\right.
		\quad \hbox{ and } \quad 
		\tilde z(t,x) = 0 \quad \hbox{ for }(t,x) \in (0,T) \times \omega. 
\end{equation*}
Using then the classical unique continuation property for the heat equation, we deduce that $\tilde z = 0$ in $(0,T) \times \Omega$, so that $\partial_t z = \mu z$. In particular, this implies that there exists a function $z_\mu \in H^1_0(\Omega)$ such that for all $(t,x) \in (0,T) \times \Omega$, $z(t,x) = e^{\mu t} z_{\mu}(x)$. According to \eqref{HeatEq-Adj-Ex3}, we then deduce
\begin{equation}
	\label{Stationary-mu}
		\left\{
			\begin{array}{ll}
				(\mu +\Delta) z_\mu = a_\mu w_\mu, \quad &  \hbox{ for }x\in \Omega, 
				\\
				z_\mu (x) = 0, \quad &  \hbox{ for }x \in \partial \Omega, 
			\end{array}
		\right.
		\quad \hbox{ and } \quad 
		z_\mu(x) = 0 \quad \hbox{ for }  x \in \omega. 
\end{equation}
Therefore, to conclude the argument, \emph{i.e.} to deduce that $z_\mu = 0$ in $\Omega$, we need to assume that $w_\mu$ is such that 
there are no solution $Z_\mu$ of 
\begin{equation}
	\label{Stationary-z-mu}
		\left\{
			\begin{array}{ll}
				(\mu +\Delta) Z_\mu =  w_\mu, \quad &  \hbox{ for }x\in \Omega, 
				\\
				Z_\mu (x) = 0, \quad &  \hbox{ for }x \in \partial \Omega, 
			\end{array}
		\right.
\end{equation}
which satisfies $Z_\mu = 0$ in $\omega$. 

In our setting, we shall thus distinguish two cases depending if $\mu$ belongs to the spectrum of the Laplacian or not. Let us denote by $\lambda_j$ the family of eigenvalues of the Laplace operator $A = - \Delta$ defined on $H= L^2(\Omega)$ with domain $\mathscr{D}(A) = H^2 \cap H^1_0(\Omega)$, indexed in increasing order $0 < \lambda_1 \leq \lambda_2 \leq \cdots \leq \lambda_j \leq \to \infty$, and $H_j = \hbox{Ker\,}(\Delta + \lambda_j)$ the corresponding eigenspace.

If $\mu \notin \{ \lambda_j,\ j \in \N\}$, then the solution of \eqref{Stationary-z-mu} is unique and therefore we shall ask that $\| Z_\mu \|_{L^2(\omega)} \neq 0$.

If $\mu= \lambda_j$, then according to the Fredholm alternative (see e.g. \cite[Appendix D, Section D.5]{EvansPDE}),
\begin{itemize}
	\item If $\P_{H_j} w_\mu \neq 0$, then there is no solution $Z_\mu$ of \eqref{Stationary-z-mu} (where $\P_{H_j} $ is the orthogonal projection on $H_j$).
	\item If $\P_{H_j} w_\mu = 0$, then any solution $Z_\mu$ of \eqref{Stationary-z-mu} writes $Z_\mu^* + \Phi_j$, where $\P_{H_j} (Z_\mu^* ) = 0$ and $Z_\mu^*$ solves \eqref{Stationary-z-mu}, and with $\Phi_j \in H_j$. Therefore, to guarantee that there are no solution $Z_\mu$ of \eqref{Stationary-z-mu} such that $Z_\mu = 0$ in $\omega$, we should assume that
	$$
	\inf_{\Phi_j \in H_j} \norm{ Z_\mu^* +\Phi_j}_{L^2(\omega)} >0.
	$$
\end{itemize}

In the above example, we made the choice of presenting what happens when $\G = \{0\}$ and $\W = \{ e^{\mu_t} w_\mu\}$, but in fact the strategy developed is much more general. 

Namely, we have the following result:

\begin{theorem}
	\label{Thm-Ex3}
		Assume (H1)--(H2), and let $A$ be the generator of an analytic semigroup on $H$. 
		
		Let $K \in \N$, $(\mu_k)_{k \in \{1, \cdots, K\}}$ be a family of real numbers two by two distinct, $\mathcal{W}_k$ be a family of closed vector spaces included in $H$ such that 
		\begin{equation}
			\label{Abstract-z-mu-k}
			\text{Any function $z$ satisfying } (\mu_k + A^*) z  \in \mathcal{W}_k \ \hbox{ and } \ B^* z = 0
			\ \text{ vanishes identically}, 
		\end{equation}
		and set 
		\begin{equation*}
			\W = \hbox{Span\,}\{ e^{\mu_k t} w_k, \, \hbox{ for } k \in \{1, \cdots, K\},  \hbox{ and } w_k \in \mathcal{W}_k \}.
		\end{equation*}
		
		Let $J \in \N$, $(\rho_j)_{j \in \{1, \cdots, J\}}$ be a family of real numbers two by two distinct, $\mathcal{G}_j$ be a family of closed vector spaces included in $U$ such that		
		\begin{equation}
			\label{Abstract-z-rho-j}
			\text{Any function $z$ satisfying } (\rho_j + A^*) z = 0 \ \hbox{ and } \ B^* z \in \mathcal{G}_j
			\ \text{ vanishes identically,} 
		\end{equation}
		and set 
		\begin{equation*}
			\G = \hbox{Span\,}\{ e^{\rho_j t} g_j, \, \hbox{ for } j \in \{1, \cdots, J\},  \hbox{ and } g_j \in \mathcal{G}_j \}.
		\end{equation*}

		We also assume that 
		\begin{equation}
			\label{EmptyCap}
			\{ \mu_k, \, k \in \{1, \cdots, K\} \} \cap \{ \rho_j, \, j \in \{1, \cdots, J\} \} = \emptyset.
		\end{equation}
		
		Finally, we also assume that the classical unique continuation property \eqref{UC-Classical} holds.

		Then the unique continuation property \eqref{UC} is satisfied. 
\end{theorem}
Before going into the proof of Theorem \ref{Thm-Ex3}, let us comment the assumptions. 

Assumption \eqref{Abstract-z-mu-k} is similar to the requirement given in the above example that any $z$ solving \eqref{Stationary-mu} should vanish identically. In fact, as stated above, if $\mathcal{W}_k$ is a one-dimensional vector space $\hbox{Span\,} (w_k)$ and $\mu_k$ does not belong to the spectrum of $-A^*$, then one only has to check that the solution $z_k$ of $(\mu + A^*) z_k = w_k$ satisfies $\norm{B^* z_k}_U \neq 0$. 

Assumption \eqref{Abstract-z-rho-j} is slightly different. In particular, if $\rho_j$ does not belong to the spectrum of $-A^*$, there are no non-trivial solution of $(\rho_j + A^*) z = 0$ so that condition \eqref{Abstract-z-rho-j} is automatically satisfied. 

Finally, the condition \eqref{EmptyCap} is there to avoid coupling of the elements of $\G$ and $\W$. Otherwise, it could be replaced by the following condition: If $\nu \in \{ \mu_k, \, k \in \{1, \cdots, K\} \} \cap \{ \rho_j, \, j \in \{1, \cdots, J\} \} $, then we write $\nu= \mu_k = \rho_j$ and we should assume that
$$
		\text{Any function $z$ satisfying } (\nu + A^*) z  \in \mathcal{W}_k \ \hbox{ and } \ B^* z \in \mathcal{G}_j
		\ \text{ vanishes identically}. 
$$

Let us also comment that we assumed in Theorem \ref{Thm-Ex3} that all the coefficients are real because we are thinking at real vector spaces, but the coefficients $\mu_k$ and $\rho_j$ can of course be taken as complex numbers provided we endowed our functional setting with a complex structure.

\begin{proof}
	We consider a solution $z$ of \eqref{AdjointEq} with some source term $w \in \W$, satisfying \eqref{LinearRelation} for some $g \in \G$. 

	In order to prove Theorem \ref{Thm-Ex3}, the basic idea is to apply the operator 
	$$
		P = \prod_{k = 1}^K (\partial_t - \mu_k) \prod_{j = 1}^J (\partial_t - \rho_j)
	$$
	to $z$, since for any $w \in \W$ and $g \in \G$, $P w = 0$ and $P g = 0$. 

	To justify this computation, we need to guarantee that $z$ has some nice time regularity properties. This is guaranteed by the fact that the semigroup of generator $A$ is analytic, so that $z$ is in fact analytic in time on $[0,T)$ (recall that the source term of $z$ solution of \eqref{AdjointEq} is an element of $\W$, which is also analytic in time).
	
	Therefore, one should have that 
	$$
		(Pz)' + A^* (Pz) = 0 \quad t \in (0,T),\quad \hbox{ and } \quad B^* Pz = 0 \hbox{ on } (0,T), 
	$$
	while $Pz \in C^0([0,T), H)$. Using that $A$ (thus $A^*$) generates an analytic semi-group, we deduce that for all $T' \in (0,T)$, $\tilde p$ given as the solution of 
	$$
		\tilde p' + A^* \tilde p = 0 \quad  \hbox{ for } t \in (0,T), \quad \hbox{ with } \quad \tilde p(T) = Pz(T'), 
	$$ 
	satisfies $B^* \tilde p = 0$ in $(0,T)$. Hence the unique continuation condition \eqref{UC-Classical} implies that $\tilde p = 0$ on $(0,T)$, and in particular that $Pz = 0$ in $[0,T')$. As $T'$ is arbitrary in $(0,T)$, $P z = 0$ in $[0,T)$.
		
	It then follows that there exists $(z_k)_{k \in \{1, \cdots, K\}}$ and $(\tilde z_j)_{j \in \{1, \cdots, J\} } $ in $H$ such that for $t \in [0,T)$,
	$$
		z(t) = \sum_{k = 1}^K z_k e^{\mu_k t} + \sum_{j = 1}^J \tilde z_j e^{\rho_j t}.
	$$
	Now, writing that $z$ solves \eqref{AdjointEq} with some source term $w \in \W$, and satisfies \eqref{LinearRelation} with some $g \in \G$, using the definitions of $\W$ and $\G$ and the independence of the family $(e^{\mu_k t})_{ k \in \{1, \cdots, K\} }, (e^{\rho_j t})_{j \in \{1, \cdots, J\} }$, we see that each $z_k$ should satisfy
	$$
		(\mu_k + A^* ) z_k \in \W_k, \quad \hbox{ and } \quad B^* z_k = 0, 
	$$ 
	while each $\tilde z_j$ should satisfy
	$$
		(\rho_j + A^*) \tilde z_j = 0, \quad \hbox{ and } \quad B^* \tilde z_j \in \G_j. 
	$$
	Conditions \eqref{Abstract-z-mu-k} and \eqref{Abstract-z-rho-j} then imply that $z$ vanishes identically, and thus that $z_T = 0$, $g = 0$ and $w = 0$. This proves the unique continuation property \eqref{UC}.
\end{proof}

\begin{remark}
	In fact, the above proof works as well if $\W$ and $\G$ are such that there exists an interval of time $(T_1, T_2) \subset [0,T]$ such that 
	\begin{align*}
		&\W = \hbox{Span\,}\{ f_k(t) w_k, \, k \in \{1, \cdots, K\},\, w_k \in \mathcal{W}_k, \,  f_k(t) = e^{\mu_k t} \hbox{ for } t \in ( T_1, T_2) \}, 
		\\
		& \G= \hbox{Span\,}\{ f_j(t) g_j, \,  j \in \{1, \cdots, J\},\, g_j \in \mathcal{G}_j, \,  f_j(t) = e^{\rho_j t} \hbox{ for } t \in ( T_1, T_2) \}, 
	\end{align*}
	with $\mu_k$, $\mathcal{W}_k$, $\rho_j$ and $\mathcal{G}_j$ as in Theorem \ref{Thm-Ex3}.
\end{remark}
%
%%%%%%%%%%%%%%%%%%%
%

\section{Further comments}\label{Sec-Further}
We would like to end up this article with a number of comments. 

\subsection{Fenchel Rockafellar theorem}
In fact, our results and proofs can be fit into the framework of Fenchel Rockafellar convex-duality theory \cite{Rockafellar-1967}, \cite{Lions-1992}. Using this theory, the minimization problem of the functional $J_{\text{ap}}$ we consider in the proof of Theorem \ref{Thm-ApproxControl} can be interpreted as the dual problem of the one which consists in minimizing 
\begin{multline*}
%	\label{Def-J-direct}
	J(u ) = \frac{1}{2} \int_0^T \| y(t)\|_{H}^2 \, dt +\frac{1}{2} \int_0^T \| u(t) \|_{U}^2 \, dt, 
	\\
	\hbox{ among all controls } u \in L^2(0,T; U), \hbox{such that the solution $y$ of  \eqref{ControlledEq} satisfies \eqref{ApproxCont-Req}, \eqref{Moments-Of-Y-in-W} and \eqref{Proj-Of-Y-T-in-E}.}  
\end{multline*}
Similarly, the functional introduced in the proof of Theorem  \ref{Thm-ExactCont}, respectivelty Theorem \ref{Thm-NullCont}, can be interpreted  as the dual functionals corresponding to the minimization of 
$$
	 \frac{1}{2} \int_0^T \| y(t)\|_{H}^2 \, dt +\frac{1}{2} \int_0^T \| u(t) \|_{U}^2 \, dt, 
$$
among all controls $u \in L^2(0,T; U)$ such that the corresponding solution $y$ satisfies all the requirements of Theorem \ref{Thm-ExactCont}, respectively Theorem \ref{Thm-NullCont}. 

In fact, this duality theory is very helpful to reduce the proof of control results to suitable unique continuation and observability properties for the adjoint equation. We refer for instance to the recent work \cite{TrelatWangXu} for developments related to stabilization properties.

\subsection{Operators with time variable coefficients}

Here, for sake of simplicity, we consider a controlled equation given by \eqref{ControlledEq} with coefficients which are independent of time. But this is in fact not really needed and we can consider equations which writes as $y' = A(t) y + B u$ for $t \in (0,T)$ provided it can be endowed with a suitable functional setting.  

For instance, if we consider a heat type equation 
\begin{equation*}
%	\label{HeatEq-variable-a0}
		\left\{
			\begin{array}{ll}
				\partial_{t} y - \Delta y +a_0(t,x) y= u \chi_\omega, \quad & \hbox{ for } (t,x) \in (0,T) \times \Omega, 
				\\
				y(t,x) = 0, \quad &  \hbox{ for }(t,x) \in (0,T) \times \partial \Omega, 
				\\
				y(0, \cdot) = y_0, & \hbox{in } \Omega, 
			\end{array}
		\right.
\end{equation*}
for some $a_0 = a_0(t,x) \in L^\infty((0,T) \times \Omega)$, the same strategy as above applies immediately. In particular, the relevant unique continuation property (corresponding to \eqref{UC}) writes as follows: if $z$ satisfies
\begin{equation}
	\label{HeatEq-Adj-Sec6}
		\left\{
			\begin{array}{ll}
				\partial_{t} z + \Delta z = a_0 (t,x) z + w, \quad &  \hbox{ for }(t,x) \in (0,T) \times \Omega, 
				\\
				z(t,x) = 0, \quad &  \hbox{ for }(t,x) \in (0,T) \times \partial \Omega, 
				\\
				z(T, \cdot) = z_T, & \hbox{in } \Omega, 
			\end{array}
		\right.
\end{equation}
with $z_T \in L^2(\Omega)$ and $w \in \W$, and 
\begin{equation}
	\label{UC-heat-Sec6}
	z(t,x) = g(t,x) \quad \hbox{ in } (0,T) \times \omega, 
\end{equation}
for some $g \in \G$, then 
\begin{equation}
	\label{UC-Sec6}
	z_T = 0, \qquad g = 0, \qquad w = 0.
\end{equation}

It is also clear under this form that the approach proposed in Section \ref{Subsec-Ex3} will not apply easily in such settings. 

In fact, we refer to \cite{Mophou-Nakoulima-2008,Mophou-Nakoulima-2009} for examples in which time-dependent potentials are considered, and even semi-linear parabolic systems, under conditions which correspond to the ones given in Section \ref{Subsec-Ex2}. According to the above remark, these are rather natural conditions to ensure the unique continuation property \eqref{HeatEq-Adj-Sec6}--\eqref{UC-heat-Sec6}--\eqref{UC-Sec6}. 

\subsection{More general settings}

It would be interesting to further develop the approach presented here to a more general abstract setting. One might consider for instance the case of unbounded control operator. 

In fact, it is quite clear that provided $B$ is an admissible control operator, in the sense of \cite[Chapter 4]{TWBook}, the proofs of Theorem \ref{Thm-ApproxControl}, Theorem \ref{Thm-ExactCont}, and Theorem \ref{Thm-NullCont} apply without any change, since the functionals $J_{\text{ap}}$ in \eqref{Functional-Approx-Cont}, $J_{\text{ex}}$ in \eqref{Functional-Exact-Cont} are still continuous on $H \times \G \times \W \times L^2(0,T; H)$ in this case, while $J_{\text{nu}}$ in \eqref{Functional-Null-Cont} can be extended to $X_{obs}$ in \eqref{X-Obs} as before.

One could also try to adapt our results in a non-hilbertian framework and rather deal with Banach spaces, similarly as in \cite{FabPuelZua} in the context of parabolic equations, since this is sometimes more appropriate depending on the equations under considerations. 

It would also be interesting, as suggested by G\"unther Leugering, who I hereby thank, to address similar questions with constraints on the control space restricting the control to be at all time in some bounded convex set (for instance balls), similarly as what is done in \cite{Ahmed-1985}. This question is of major interest for practical applications of control theory.

\subsection{Cost of controllability in $\G$ and $\W$}\label{Subsec-Cost}
It would be interesting to try to quantify the cost of controlling the projections on $\G$ and $\W$ in the spirit of what has been done in \cite{FernandezCaraZuazua1}. There, the cost of the approximate controllability problem \eqref{ApproxCont-Req}, \eqref{Proj-Of-Y-T-in-E}  for some finite dimensional space $E$ (that is, corresponding to $\G = \{0\}$ and $\W = \{0\}$) was discussed for heat equations. 

\subsection{Moment approaches}
In the special case in which the equation \eqref{ControlledEq} is the one-dimensional heat equation on an interval $\Omega = (0,1)$ with Dirichlet boundary conditions controlled at one end, for which the control $u$ is looked for in $L^2(0,T)$ (that is, $U = \R$) and $\G = \hbox{Span\,}\{t \mapsto e^{\mu t} \}$, $\W = \{0\}$, one could  solve the null-controllability problem \eqref{Moments-Of-U-in-G}--\eqref{NullCont-Req} using a moment approach similar to the one developed in \cite{FattoriniRussel71}. However, this becomes less clear when $\G = \hbox{Span\,} \{ g = g(t)\}$ for some arbitrary function $g$ in $L^2(0,T)$ not necessarily of the form of an exponential. It would be interesting to try to develop a moment approach which would recover the results in Theorem \ref{Thm-NullCont} in such case.

In fact, it would be tempting to similarly address the exact controllability problem \eqref{Moments-Of-U-in-G}--\eqref{ExactCont-Req} for $1$-d wave equation using Ingham's type argument \cite{Ing}. But here again when $\G = \hbox{Span\,} \{ g(t)\}$ for some arbitrary function $g$ in $L^2(0,T)$ not necessarily exponential, this does not seem so clear either.

Actually, we are not aware in the literature of a proof of approximate controllability (\emph{i.e.} of controls such that the trajectory $y$ of \eqref{ControlledEq} satisfies \eqref{ApproxCont-Req}) which is based on the  construction of the control through a moment approach, for instance in the case of the heat equation on a half line controlled at the boundary. In fact, even if such an approach could be developed, it is not clear how it could be adapted to solve the approximate controllability problem \eqref{Moments-Of-U-in-G}--\eqref{ApproxCont-Req} for $\G = \hbox{Span\,}\{g = g(t) \}$ with $g$ not of the form of an exponential. It would be very interesting to develop such approaches, in particular to precisely address the cost of controlling the projection on $\G$ and $\W$ mentioned in Section \ref{Subsec-Cost}.

%
%\begin{itemize}
%	\item Fenchel Rockafellar theorem \cite{Rockafellar-1967}, \cite{Lions-1992}
%	\item $B$ unbounded control operator, cf Fabre Puel Zuazua.
%	\item Approche par moments ?
%	\item Operators with time variable coefficients. Essentiellement pareil. Opérateurs non-linéaire. 
%	\item Approximation des projections.
%\end{itemize}

%
%%%%%%%%%%%%%%%%%%%%%%%%%%%%%%%%%%%%%%%%%%%
%
\bibliographystyle{plain}
%\bibliography{/Users/ervedoza/Desktop/Biblio} 

%
%
\end{document}